\documentclass[letterpaper,10pt]{amsart}
\usepackage{indentfirst} 
\usepackage{amssymb}
\usepackage{amsmath}
\usepackage{mathabx} 
\usepackage{amsthm}
\usepackage{thmtools}
\usepackage{bbm}             
\usepackage[colorlinks=true]{hyperref}  
\usepackage[usenames,dvipsnames]{xcolor} 
\usepackage{tikz}
\usepackage{comment}

\newtheorem{theorem}{Theorem}[section]
\newtheorem{definition}[theorem]{Definition}
\newtheorem{proposition}[theorem]{Proposition}
\newtheorem{corollary}[theorem]{Corollary}
\newtheorem{lemma}[theorem]{Lemma}

\newtheorem{remark}[theorem]{Remark}
\newtheorem{conjecture}[theorem]{Conjecture}
\newtheorem*{theorem-non}{Theorem}

\theoremstyle{definition}
\newtheorem{example}[theorem]{Example}

\newcommand{\M}{\mathcal{M}}
\newcommand{\R}{\mathbb{R}}
\newcommand{\Z}{\mathbb{Z}}
\newcommand{\N}{\mathbb{N}}

\renewcommand{\P}{\mathbb{P}}

\newcommand{\cE}{\mathcal{E}}
\newcommand{\cH}{\mathcal{H}}
\newcommand{\cK}{\mathcal{K}}
\newcommand{\cM}{\mathcal{M}}

\newcommand{\cU}{\mathcal{U}}
\newcommand{\cV}{\mathcal{V}}

\def\phi{\varphi}
\def\R{{\mathbb R}}
\def\C{{\mathbb C}}
\def\N{{\mathbb N}}
\def\Z{{\mathbb Z}}

\def\H{{\mathbb H}}

\def\P{{\mathcal P}}
\def\Q{{\mathcal Q}}

\def\M{{\mathcal M}}

\def\S{{\mathcal S}}

\def\supp{\mbox{\rm supp}}

\def\le{\leqslant}
\def\ge{\geqslant}

\def\e{\epsilon}

\def\M{\mathcal{M}}

\begin{document}

\title{Thermodynamic formalism and the entropy at infinity of the geodesic flow}

\author[A. Velozo]{Anibal Velozo}  \address{Department of Mathematics, Yale University, New Haven, CT 06511, USA.}
\email{\href{anibal.velozo@gmail.com}{anibal.velozo@yale.edu}}
\urladdr{\href{https://gauss.math.yale.edu/~av578/}{https://gauss.math.yale.edu/~av578/}}

\begin{abstract}
In this paper we study the ergodic theory and thermodynamic formalism  of the geodesic flow on non-compact pinched negatively curved manifolds. We consider two notions of entropy at infinity, the topological and the measure theoretic entropy at infinity. We prove that  both notions coincide and we relate this quantity with the upper semicontinuity of the entropy map. This relationship was already studied in \cite{rv} for geometrically finite manifolds; in this paper we extend such results to arbitrary pinched negatively curved manifolds. We obtain several applications to the existence and stability of equilibrium states. Our approach has the advantage to include certain uniformly continous potentials, a situation that can not be studied through the methods from \cite{pps}. 
Of particular importance are potentials that vanish at infinity and those that satisfy a critical gap condition, that is, strongly positive recurrent potentials. We also prove some equidistribution and counting results for closed geodesics (beyond the geometrically finite case). In particular we obtain a version of the prime geodesic theorem for strongly positive recurrent manifolds. 
  
\end{abstract}

\maketitle

\section{Introduction}

In this paper we will study the ergodic theory and the thermodynamic formalism of the geodesic flow on non-compact  pinched negatively curved manifolds. In order to put our results into context  we start with a brief discussion on what is known in the compact case, which is better understood than the non-compact case and motivates a large part of the results in this paper.

 By the work of Bowen \cite{b73} and Ratner \cite{rat}, we know that the geodesic flow on a compact negatively curved manifold can be modelled as a suspension flow over a sub-shift of finite type. This symbolic representation has many important consequences for the thermodynamic formalism of H\"older potentials: the topological pressure can be computed using a weighted sum over the periodic orbits,  the equilibrium states are unique and satisfy the Gibbs property, the central limit theorem holds, and large deviation estimates are available (see \cite{bob}, \cite{rat2}, \cite{pol}). Another important property of the geodesic flow on a compact negatively curved manifold is the  upper semicontinuity of the entropy map \cite{bo1}. In the compact case the upper semicontinuity of the entropy map implies the existence of a measure of maximal entropy. In the same context, the measure of maximal entropy was  constructed independently by Margulis \cite{marg} and Bowen \cite{b73}.

Unfortunately, if the ambient manifold is non-compact then most of this discussion does not apply:  we do not have  at our disposal a symbolic representation, and the proof of the upper semicontinuity of the entropy map does not work. Let me be a bit more precise about these two technical difficulties. First, in the study of non-compact manifolds we are  often interested in situations where the injectivity radius is not bounded below (for instance in the presence of `cusps')--this is a very problematic issue when trying to extend the symbolic model of Bowen and Ratner to the non-compact case. Second, the proof of the upper semicontinuity of the entropy map is based on two hypotheses: the compactness of the ambient space, and that the dynamical system is expansive (or some weaker assumption like h-expansive, or asymptotically h-expansive).  If one removes the compactness assumption then the upper semicontinuity of the entropy map does not necessarily hold, even when the dynamical system is expansive. As an example, consider the full shift on a countable (infinite) alphabet. In this case, the entropy map is not upper semicontinuous at any measure of finite entropy. 

Fortunately it is not all bad news: in a recent work Paulin, Pollicott and Schapira \cite{pps} were able to extend many of the results known to hold for H\"older potentials on compact negatively curved manifolds to the non-compact case; this generalizes the substantial work of Roblin  \cite{ro}. More precisely, they proved the uniqueness of equilibrium states and the  equivalence between the topological pressure, the Gurevich pressure, and the critical exponent for  H\"older potentials (see Theorem \ref{pps}). A very good complement to  \cite{pps} is the recent paper of Pit and Schapira \cite{sp} where they provided a strong criterion for the existence of equilibrium states for H\"older potentials; in particular for the existence of the measure of maximal entropy for the geodesic flow. In this work we will continue the study of the thermodynamic formalism of H\"older potentials, but we will also be interested in (uniformly) continuous potentials. In the continuous category the techniques developed in \cite{pps} do no apply. 

It worth mentioning that Sullivan \cite{su0} constructed an invariant measure that whenever finite is the unique measure of maximal entropy, up to normalization (see \cite{op}); the so-called Bowen-Margulis-Sullivan measure (BMS measure for short). We emphasize that in the non-compact case the BMS measure can be infinite, in which case there is no measure of maximal entropy. For instance, any infinite normal cover of a compact manifold has infinite BMS measure. Peign\'e \cite{pei} constructed geometrically finite manifolds with infinite BMS measure.

Before stating some of the main results of this paper let us introduce some notation. The dynamical system of interest will always be the geodesic flow on a pinched negatively curved manifold.  In particular, whenever we say `invariant measure', we  mean `non-negative Borel measure invariant by the geodesic flow'. The measure theoretic entropy of an invariant probability measure $\mu$ is denoted by $h_\mu(g)$.  We say that $(\mu_n)_n$ converges vaguely to $\mu$ if for every compactly supported continuous function $f$ we have that $\lim_{n\to\infty}\int fd\mu_n=\int f d\mu$. Similarly, we say that $(\mu_n)_n$ converges in the weak* topology to $\mu$ if for every bounded continuous function $f$ we have that $\lim_{n\to\infty}\int fd\mu_n=\int f d\mu$. Since $M$ is non-compact, the vague limit of a sequence of probability measures might not be a probability measure. In other words, a sequence of probability measures might lose mass. The mass of a measure $\mu$ is the number $\mu(T^1M)$, and it is denoted by $|\mu|$. A standing assumption in this paper is that $(M,g)$ is non-elementary (see Section \ref{strucneg}), we will assume this is the case from now on. A fundamental quantity in this paper is the topological entropy at infinity of the geodesic flow (see Definition \ref{teidef}), which we denote by $\delta_\infty$. The following result is proved in Section \ref{5}.
\begin{theorem}[=Theorem \ref{A}]\label{i1}  Let $(M,g)$ be a pinched negatively curved manifold. Let $(\mu_n)_{n}$ be a sequence of invariant probability measures converging to $\mu$ in the vague topology. Then 
$$\limsup_{n\to \infty}h_{\mu_n}(g)\le |\mu|h_{\mu/|\mu|}(g)+(1-|\mu|)\delta_\infty.$$
If the sequence converges vaguely to the zero measure, then the right hand side is understood as $\delta_\infty$. 
\end{theorem}
A simple consequence of Theorem \ref{i1} is the upper semicontinuity of the entropy map. A result similar to Theorem \ref{i1} was previously obtained by Einsiedler, Kadyrov and Pohl in the context of homogeneous dynamics (see \cite{ekp}). Later on, Kadyrov and Pohl proved that the formula obtained in \cite{ekp} is sharp (see \cite{kp}). In a joint work with F. Riquelme, we adapted the methods used in \cite{ekp} to obtain Theorem \ref{i1} for geometrically finite manifolds (see \cite{rv}). In this paper we extend the results from \cite{rv} to arbitrary pinched negatively curved manifolds. 

In the geometrically finite case there exists a compact set that intersects every closed geodesic--this is a crucial fact used in \cite{rv}. On the other hand, if $M$ is not geometrically finite, then there exists a sequence of closed geodesics that escape every compact subset of $M$. This result was recently obtained by  Kapovich and Liu, generalizing a famous result of Bonahon in the context of  hyperbolic 3-manifolds (see \cite{kl} and \cite{bon}). The existence of closed geodesics escaping every compact set is an obstruction to be able to run the strategy  used in \cite{rv}. To overcome this difficulty we will proceed in two steps. We first assume that the measures $(\mu_n)_n$ are ergodic and give positive mass to a fixed compact subset of $T^1M$.  Under these assumptions we can prove an estimate similar to Theorem \ref{i1}, where instead of $\delta_\infty$ we have some constant that depends on the compact subset (see Proposition \ref{pree}). We then use the weak entropy density  of the geodesic flow (see Definition \ref{weddef}) to obtain Theorem \ref{i1}. The way we partition our space is also different from the ones used in previous works. We consider a partition with more dynamical meaning than those in \cite{ekp} or \cite{rv}; we decompose our space using return times to a compact piece.


The strategy used to prove Theorem \ref{i1} applies to a larger class of dynamical systems, in particular to the geodesic flow on non-positively curved rank one manifolds and countable Markov shifts of finite topological entropy. For completeness we prove a more general result than Theorem \ref{i1} which emphasizes the main hypotheses we need to verify in order to obtain similar bounds (see Theorem \ref{teoB}). In a joint work with G. Iommi and M. Todd we have successfully used these methods to understand the entropy map of countable Markov shifts with finite topological entropy (see \cite{itv}).


The topological entropy at infinity has a measure theoretic counterpart; this is what we call the measure theoretic entropy at infinity. Since its definition does not require to introduce a lot of notation we provide it here.  The  measure theoretic entropy at infinity of the geodesic flow is defined as 
$$h_\infty=\sup_{(\mu_n)_n}\limsup_{n\to \infty} h_{\mu_n}(g),$$
where the supremum runs over sequences of invariant probability measures converging vaguely to the zero measure. We emphasize that this definition make sense for any dynamical system on a locally compact metric space. We will prove a type of variational principle for the entropy at infinity. 
\begin{theorem}[=Theorem \ref{varpri}]\label{i2} The topological entropy at infinity is equal to the measure theoretic entropy at infinity. In other words $\delta_\infty=h_\infty$. 
\end{theorem}
In the context of the geodesic flow the measure theoretic entropy at infinity was introduced in \cite{irv} and proved to be equal to the topological entropy at infinity for extended Schottky manifolds via symbolic methods. This result was later extended to cover all geometrically finite manifolds in \cite{rv}. In the context of countable Markov shifts the entropy at infinity was already studied by Ruette \cite{ru} and Buzzi \cite{bu2}.

In this paper we will constantly refer to (bounded) real valued functions with domain $T^1M$ as `potentials'; this is a standard convention in thermodynamic formalism. The topological pressure of a potential $F$ is the quantity $$P(F)=\sup_{\mu\in \M(g)} \{h_\mu(g)+\int Fd\mu\},$$
where the supremum runs over the space of invariant probability measures $\M(g)$. An invariant probability measure satisfying $P(F)=h_\mu(g)+\int Fd\mu$, is called an equilibrium state for the potential $F$.

We denote by $C_0(T^1M)$  the space of continuous potentials  vanishing at infinity (see Definition \ref{limD}). Among potentials vanishing at infinity we will be particularly interested in strongly positive recurrent potentials (SPR for short); those are potentials satisfying the critical gap condition $P(F)>\delta_\infty$ (for the general definition of SPR potentials see Definition \ref{sprdef}). 

One of the main features of SPR potentials is part (\ref{i31}) in our next result--in the compact case this is an immediate consequence of the upper semicontinuity of the entropy map. Parts (\ref{i32}) and (\ref{i33}) exhibit some big differences between the compact and the non-compact situation.

\begin{theorem}[=Theorem  \ref{eeq2}]\label{i3} Let $F\in C_0(T^1M)$ be a potential that vanishes at infinity. Let $(\mu_n)_n$ be a sequence of invariant probability measures  such that $$\lim_{n\to\infty}\big(h_{\mu_n}(g)+\int Fd\mu_n\big)=P(F).$$ Then the following statements hold. 
\begin{enumerate}
\item\label{i31} If $F$ is SPR, then every convergent subsequence  (in the vague topology)  of $(\mu_n)_{n}$  converges in the weak* topology to an equilibrium state of $F$. If $F$ is H\"older continuous then $(\mu_n)_n$ converges in the weak*  topology to the equilibrium state of $F$.
\item\label{i32} Suppose that $F$ does not admit any equilibrium state.  Then  $(\mu_n)_{n}$  converges vaguely to the zero measure. In this case we  have  $P(F)=\delta_\infty$. 
\item\label{i33}  Suppose that $F$ does admit an equilibrium state.  Then the accumulation points of  $(\mu_n)_{n}$  lies in the set $$\{t\mu:t\in [0,1]\text{ and }\mu\text{ is an equilibrium state of $F$}\}.$$
\end{enumerate}
\end{theorem}

This result was obtained in \cite{v} in the context of geometrically finite manifolds and in \cite{gs} for countable Markov shifts with finite range potentials. We emphasize that we require no higher regularity on $F$ than continuity (for $F\in C_0(T^1M)$ this implies uniform continuity). 
Theorem \ref{i3} gives a characterization of SPR potentials: a potential $F\in C_0(T^1M)$ is SPR if  and only if every sequence of measures satisfying the hypothesis of Theorem \ref{i3} does not converge vaguely to the zero measure. Theorem \ref{i3} also implies that  a SPR potential admits an equilibrium state.

For H\"older SPR potentials we can compute the first derivative of the pressure (see Theorem \ref{deriva}) and prove the continuity of equilibrium states (see Theorem \ref{conteq}). We will prove in Section \ref{7.2} that SPR potentials are open and dense in $C_0(T^1M)$. That is, most potentials in $C_0(T^1M)$ are similar to those in compact negatively curved manifolds.


The relationship between the upper semicontinuity of the entropy map and  large deviations  is well understood in the compact case (see \cite{kif}, \cite{pol}). In Section \ref{ld} we use this relationship to obtain some equidistribution results we proceed to explain. 

A periodic orbit of the geodesic flow $\gamma$ determines an invariant probability measure $\mu_\gamma$. The length of $\gamma$ is denoted by $l(\gamma)$. Let $W$ be an open relatively compact subset of $T^1M$ that intersects the non-wandering set of the geodesic flow. Define 
$$\M_p(W,t)=\{\mu_\gamma: \gamma\text{ a periodic orbit such that }l(\gamma)\le t\text{ and }\mu_\gamma(W)>0\}.$$

 \begin{theorem}[=Theorem \ref{equi2}]\label{equi} Let $F\in C_0(T^1M)$ be a H\"older SPR potential such that $P(F)>0$. Then the sequence of measures (parametrized by $t$)
$$\frac{1}{\sum_{\mu_\gamma\in \M_p(W,t)}\exp(l(\gamma)\int F d\mu_\gamma)}\sum_{\mu_\gamma\in \M_p(W,t)}\exp(l(\gamma)\int F d\mu_\gamma)\mu_\gamma,$$ 
converges to the equilibrium state of $F$ in the weak* topology as $t$ goes to infinity. 
\end{theorem}

We finish with a counting result for closed geodesics intersecting a compact piece of $T^1M$. For $F=0$ this was proven by  Margulis \cite{marg} in the compact case, and by Roblin \cite{ro} when  $M$ is geometrically finite. Our result should be compared \cite[Corollary 1.7]{pps}, where Roblin's result was extended to include arbitrary H\"older potentials. In Theorem \ref{fin1} we are able to obtain a result beyong the geometrically finite setting, but under certain assumptions on the potential. The believe the assumption $F\in C_0(T^1M)$ can be removed; this would be the case if Conjecture \ref{conj1} holds. We point out that \cite[Corollary 1.7]{pps} does not require $F$ to be SPR, nor in $C_0(T^1M)$, but only works for geometrically finite manifolds (where the geometry of the ends of the non-wandering set can be controlled). We write $a(t)\sim b(t)$ if $\lim_{t\to\infty}a(t)/b(t)=1$.

\begin{theorem}[=Theorem \ref{fin}]\label{fin1} Assume that the geodesic flow is topologically mixing. Let $F\in C_0(T^1M)$ be a H\"older SPR potential such that $P(F)>0$. Then 
$$\sum_{\mu_\gamma\in \M_p(W,t)}e^{l(\gamma)\int F d\mu_\gamma}\sim \frac{e^{P(F)t}}{P(F)t},$$
for every open relatively compact  subset $W\subset T^1M$ that intersects the non-wandering set of the geodesic flow.
\end{theorem}

We say that $(M,g)$ is strongly positive recurrent (SPR for short) if $h_{top}(g)>\delta_\infty$ (where $h_{top}(g)$ is the topological entropy of the geodesic flow). This definition was coined in \cite{st} and fits perfectly within the analogy between the geodesic flow and countable Markov shifts. It worth pointing out that every hyperbolic geometrically finite manifold is SPR. A direct application of Theorem \ref{fin1} is the following.

\begin{theorem}[=Corollary \ref{fincor}] Let $(M,g)$ be strongly positive recurrent and assume its geodesic flow is topologically mixing. Let $W$ be an open relatively compact subset intersecting the non-wandering set of the geodesic flow. Then $$\#\{\gamma\text{ primitive periodic orbit }|\text{ } l(\gamma)\le t\text{ and }\gamma\cap W\ne \emptyset\} \sim \frac{e^{h_{top}(g)t}}{h_{top}(g)t}.$$
\end{theorem}

This is a version of the prime geodesic theorem for strongly positive recurrent manifolds. We emphasize that even within the class of strongly positive recurrent manifolds it might happen that 
$$\M_p(t)=\{\mu_\gamma:\gamma\text{ a periodic orbit such that }l(\gamma)\le t\},$$
has infinitely many elements. In particular it is meaningless to count all closed geodesics of length at most $t$ (see Example \ref{finrem}).\\


\noindent
{\bf Organization of the paper.} In Section \ref{2} we recall some basic facts about measure theory, entropy theory and negatively curved manifolds. Of particular importance is Section \ref{pretf} where we recall what is known about the thermodynamic formalism of the geodesic flow and introduce the important class of strongly positive recurrent potentials. In Section \ref{4} we prove the weak entropy density and a mild modification of the simplified entropy inequality for the geodesic flow. In Section \ref{moregen} we prove an abstract result in the spirit of Theorem \ref{i1} and in Section \ref{5.2} we establish Theorem \ref{i1}. In Section \ref{varprin} we prove Theorem \ref{i2}--the variational principle at infinity of the geodesic flow. In Section \ref{6} we apply the results obtained in previous sections to the existence and stability of equilibrium states. In Section \ref{ld} we prove the Theorem \ref{equi} and Theorem \ref{fin1}.\\

\noindent
{\bf Acknowledgements.} I would like to thank my PhD advisor Gang Tian for his constant support and encouragements. I would also like to thank Barbara Schapira for telling me about her joint work with Samuel Tapie, where they have also introduced and studied the topological entropy at infinity of the geodesic flow \cite{st}. Finally, I would like to thank Godofredo Iommi, Hee Oh and Felipe Riquelme for useful discussions about the topics of this paper, and to Yair Minsky for his valuable help with Example \ref{finrem}.

\section{Preliminaries}\label{2}
In this section we collect some important facts about the geodesic flow on a pinched negatively curved manifolds. The definitions and facts given in this section will be constantly used in following sections. Since the spaces of interest for this paper are non-compact we need to be  careful with the spaces of invariant probability measures that we use (it is a technically unpleasant to restrict all our attention to the space of invariant probability measures, which is non-compact). We start by making explicit the spaces of measures that we will use, and their topologies. 

\subsection{Measure theory }\label{measureth}
Let $(X,d)$ be a locally compact metric space and $T:X\to X$, a continuous map. In this paper whenever we say \emph{measure} we mean a non-negative countably additive Borel measure. The mass of a measure is the number $\mu(X)$ and it is denoted by $|\mu|$. We say that a measure $\mu$ is $T$-invariant if $\mu(T^{-1}A)=\mu(A)$, for every Borel set $A$. We denote by $\M(X,T)$ the space of $T$-invariant probability measures on $X$ and $\M_{\le 1}(X,T)$ the space of  $T$-invariant measures such that $|\mu|\in [0,1]$ (we also refer to them as sub-probability measures). Clearly $\M(X,T)\subset \M_{\le 1}(X,T)$. We endow the space $C_b(X)$ (resp. $C_c(X)$) of bounded  (resp. compactly supported) continuous functions with the uniform norm $||f||_0=\sup_{x\in X} |f(x)|$. We endow $\M(X,T)$ with the \emph{weak* topology}: a sequence $(\mu_n)_{n}$ converges weakly to $\mu$ if for every $f\in C_b(X)$ we have $$\lim_{n\to\infty}\int fd\mu_n=\int fd\mu.$$  In a similar way we endow $\M_{\le 1}(X,T)$ with the \emph{vague topology}: a sequence $(\mu_n)_{n}$ converges vaguely to $\mu$ if for every $f\in C_c(X)$ we have $$\lim_{n\to\infty}\int fd\mu_n=\int fd\mu.$$ 
If there exists a compact exhaustion of $X$, that is, an increasing sequence $(K_n)_{n}$ of compact sets such that $X=\bigcup_{n\ge 1}K_n$, then the space $C_c(X)$ is separable. In this case consider  a dense subset $(f_n)_{n}$ of the unit ball of $C_c(X)$. We define a metric $\rho$ on $\M_{\le 1}(X,T)$ by the formula $$\rho(\mu_1,\mu_2)=\sum_{n\ge 1}\frac{1}{2^n}|\int fd\mu_1-\int fd\mu_2|.$$
This metric is compatible with the vague topology. By Banach-Alaoglu theorem we know that $\M_{\le 1}(X,T)$ is a compact metric space. If $X$ is a non-compact manifold, then $X$ admits a compact exhaustion, and therefore this discussion applies to that case. The connection between this two topologies is given by the following simple fact (see \cite[Theorem 13.16]{kle}). 
\begin{theorem} \label{weakvag} Let $X$ be a locally compact metric space. Then for a sequence $(\mu_n)_n\subset \cM(X,T)$ the following statements are equivalent.
\begin{enumerate}
\item $(\mu_n)_n$ converges to $\mu$ in the weak* topology.
\item $(\mu_n)_n$ converges to $\mu$ in the vague topology and $\mu$ is a probability measure.
\end{enumerate}
\end{theorem}
Theorem \ref{weakvag} tells us that the only obstruction to converge in the weak* topology is the possible escape of mass. 

\subsection{Entropy theory}\label{entropyth} This theory provides tools to  measure the complexity of a dynamical system. In this paper we will be mainly interested in the following two notions of entropy.

\subsubsection{Topological entropy}\label{topentropy} Let $(X,d)$ will be a metric space and $T:X\to X$, a continuous map. 
We define the dynamical metrics $(d_n)_n$ given by the formula: $$d_n(x,y)=\max_{i\in\{0,...,n-1\}} d(T^ix,T^iy).$$
For a flow $(T_t)_t$ we define $$d_t(x,y)=\max_{s\in [0,t]}d(T_sx,T_sy).$$ We denote by $B_n(x,r)$ to the ball centered at $x$ of radius $r$ in the metric $d_n$. A ball $B_n(x,r)$ is also called a $(n,r)$-\textit{dynamical ball}. Given a compact subset $K$ of $X$ we denote by $N(K,n,r)$ to the minimum number of $(n,r)$-dynamical balls needed to cover $K$. The \emph{topological entropy} of $(X,d,T)$ is defined as
$$h_d(T)=\sup_{K\subset X} \lim_{r\to 0}\lim_{n\to\infty} \dfrac{1}{n}\log N(K,n,r),$$
where the last limit runs over compact subsets of $X$. This definition was first introduced by R. Bowen in \cite{bo1} as a way to extend the classical definition of topological entropy on compact spaces to the non-compact setting. 

\subsubsection{Measure theoretic entropy} \label{meaentropy} Given a countable measurable partition $\P=\{P_i\}_{i\in I}$ of $X$ we define the entropy of $\P$ as $$H_\mu(\P)=-\sum_{i\in I}\mu(P_i)\log\mu(P_i).$$
Given two partitions $\P$ and $\Q$ we can construct the smallest common refinement of $\P$ and $\Q$, this is denoted by $\P\vee \Q$. Define
$$h_\mu(T,\P)=\lim_{n\to\infty} \dfrac{1}{n}H(\bigvee_{i=0}^nT^{-i}\P).$$
By taking the supremum over all countable partitions of finite entropy we obtain the \textit{entropy of $T$ with respect to} $\mu$, this is denoted by $h_\mu(T)$. In other words 
$$h_\mu(T)=\sup_{\P} h_\mu(T,\P),$$
where the supremum runs over all countable partitions of finite entropy. 
\begin{definition}[Entropy map] Given a dynamical system $(X,d,T)$ we refer to the map $\mu\mapsto h_\mu(T)$, as the \emph{entropy map}. 
\end{definition}

For a flow $(T_t)_t$ the measure theoretic entropy of $\mu$ is always computed for the time one map. 

\subsubsection{Katok's entropy formula} The topological entropy and the measure theoretic entropy seem, at first, a bit unrelated. The following theorem provides a very strong connection between them--it helps to understand the measure theoretic entropy in the spirit of the topological entropy. 
\begin{theorem}\label{kat} Let $(X,d)$  be a compact metric space and $T:X\to X$, a continuous transformation. Let $\mu$ be an ergodic $T$-invariant probability measure and $\delta\in(0,1)$. Then
$$h_\mu(T)=\lim_{r\to\infty}\liminf_{n\to\infty}\dfrac{1}{n}\log N_\mu(n,r,\delta),$$
where $N_\mu(n,r,\delta)$ is the minimum number of $(n,r)$-dynamical balls needed to cover a set of $\mu$-measure at least $1-\delta$. In particular the  limit above is independent of $\delta\in (0,1)$.
\end{theorem}
This result was proven by A. Katok in \cite{ka}. If one follows the proof of Theorem \ref{kat} (see \cite[Section 1]{ka}), it is clear that the inequality $$h_\mu(T)\le \lim_{r\to\infty}\liminf_{n\to\infty}\dfrac{1}{n}\log N_\mu(n,r;\delta),$$
holds regardless of the compactness of $X$. For most of our purposes this inequality will be enough. Recently F. Riquelme \cite{r} proved that equality holds if $X$ is a topological manifold and $T$ a Lipschitz map. More precisely we have the following result. 
\begin{theorem}\label{ri} Let $(X,d)$  be a metric space and $T:X\to X$ a continuous transformation. Then for every  ergodic $T$-invariant probability measure $\mu$ we have
$$h_\mu(T)\le \lim_{r\to\infty}\liminf_{n\to\infty}\dfrac{1}{n}\log N_\mu(n,r,\delta),$$
where $N_\mu(n,r,\delta)$ as in Theorem \ref{kat}. If we moreover assume that $(X,d)$ is a topological manifold and $T$ is Lipschitz, then the equality holds. 
\end{theorem}
The following two definitions will be of great importance in this paper.

\begin{definition}[Simplified entropy formula] \label{sef} We say that $(X,d,T)$ satisfies a \emph{simplified entropy formula} if there exists $\e_0>0$ such that for every $\e\in (0,\e_0)$, for every ergodic probability measure $\mu$ and $\delta\in (0,1)$ we have
$$h_\mu(T)= \limsup_{n\to\infty} \frac{1}{n}\log N_\mu(n,\e,\delta).$$
We say that $(X,d,T)$ satisfies a \emph{simplified entropy inequality}  if there exists $\e_0>0$ such that for every $\e\in (0,\e_0)$,   for every ergodic probability measure $\mu$ and $\delta\in (0,1)$ we have 
$$h_\mu(T)\le \limsup_{n\to\infty} \frac{1}{n}\log N_\mu(n,\e,\delta).$$
\end{definition}

\begin{definition}[Weak entropy dense] \label{weddef} We say that the space of ergodic measures is \emph{weak entropy dense} in $\M(X,T)$ if the following holds: for every $\mu\in\M(X,T)$ and $\eta>0$ there exists a sequence $(\mu_n)_{n}$ of ergodic measures satisfying
\begin{enumerate}
\item $(\mu_n)_n$ converges in the weak* topology to $\mu$
\item For every $n\in \N$ we have $h_{\mu_n}(T)>h_\mu(T)-\eta$.
\end{enumerate}
We also refer to this property by saying that $(X,d,T)$ is \emph{weak entropy dense}.
\end{definition}

We will prove in Section \ref{4} that the geodesic flow on a pinched negatively curved manifold  satisfies a mild modification of the simplified entropy inequality and that its ergodic measures are weak entropy dense.

\subsection{ Structure of negatively curved manifolds}\label{strucneg}
From now on we will assume that $(M,g)$ is a complete Riemannian manifold of negative sectional curvature. We will moreover assume that $K_g\in [-a,-b]$, for some $a,b>0$, where $K_g$ is the sectional curvature of $M$. We refer to a  manifold satisfying those properties as a \emph{pinched negatively curved manifold}. 

The unit tangent bundle of $M$ is defined as $T^1M=\{v\in TM: ||v||_g=1\}$. Since $(M,g)$ is complete, the geodesic flow on $T^1M$ is well defined for all times; we denote it by $(g_t)_{t\in\R}$. The Riemannian metric $g$ makes $M$ into a metric space, the induced distance function (shortest path distance) is denoted by $d$.  Let $\pi :T^1M\to M$ be the canonical projection. We define a metric on $T^1M$--which we still denote by $d$--in the following way \begin{align}\label{distance} d(x,y)=\max_{t\in [0,1]} d(\pi g_t(x),\pi  g_t(y)),\end{align}
for every $x,y\in T^1M$. We emphasize that this is the metric used in all the statement about the geodesic flow (but other possible candidates of metrics are usually H\"older equivalent to $d$, at least in the pinched negatively curved case, see \cite[Section 2.3]{pps}).

 The space of invariant  probability (resp. sub-probability) measures of the geodesic flow is denoted by $\M(g)$ (resp. $\M_{\le1}(g)$). As explained in Section \ref{measureth} we endow $\M_{\le 1}(g)$ with the vague topology and $\M(g)$ with the weak* topology. The topological entropy of the geodesic flow is denoted by $h_{top}(g)$.

 The universal cover of $M$ is denoted by $\widetilde{M}$ and its visual boundary by $\partial_\infty\widetilde{M}$. The space $\widetilde{M}\cup\partial_\infty\widetilde{M}$ is called the Gromov compactification of $\widetilde{M}$ and it is homeomorphic to a closed ball. The projection map $T^1\widetilde{M}\to\widetilde{M}$ is still denoted by $\pi$. We denote by $Iso(\widetilde{M})$ the group of isometries on $\widetilde{M}$.  It is well known that every isometry of $\widetilde{M}$ extends to a homeomorphism of $\partial_\infty \widetilde{M}$ and that the fundamental group of $M$ acts (via Deck transformations) isometrically, freely and discontinuously on $\widetilde{M}$. After fixing a reference point we identify $\pi_1(M)$ with a subgroup $\Gamma<$ $Iso(\widetilde{M})$. The \emph{limit set} of $\Gamma$--which we denote by $L(\Gamma)$--is the set of accumulation points in $\partial_\infty\widetilde{M}$ of the orbit of a point $x\in \widetilde{M}$ under $\Gamma$. We say that $\Gamma$ is  \emph{non-elementary}, if it is not generated by one hyperbolic element, nor a parabolic subgroup (see \cite{bow} for several characterizations of this property).\\

\noindent
{\bf Standing assumption:} In this paper $\Gamma$ will always be assumed to be non-elementary. \\

In this case $L(\Gamma)$ is the minimal closed $\Gamma$-invariant subset of $\partial_\infty \widetilde{M}$ and $\#L(\Gamma)>2$.  Let $\Omega\subset T^1M$ be the \emph{non-wandering set} of the geodesic flow. 
It worth pointing out that $\Omega$ concentrates all the dynamics of the geodesic flow. By Poincar\'e recurrence, $\Omega$ contains the support of all the invariant probability measures of the geodesic flow. In this paper we will be mostly interested in the case that $\Omega$ is non-compact. We use the notation $\widetilde{\Omega}:=p^{-1}(\Omega)$ for the preimage of $\Omega$ under $p:T^1\widetilde{M}\to T^1M$. If $\Omega$ is compact, then we say that $M$ is \emph{convex-cocompact}. We say that $M$ is \emph{geometrically finite} if $L(\Gamma)$ is union of bounded parabolics and radial limit points (for a precise definition see \cite{bow}). As mentioned in the introduction it was recently proved that $M$ is geometrically finite if and only if there exists a compact set $K\subset M$ that intersects all closed geodesics \cite{kl}.

\subsection{Thermodynamic formalism}\label{pretf}
We start with the definition of the topological pressure of a potential. In this paper we will only deal with the case when $F$ is bounded; this is not essential but makes some statements easier to write. 

\begin{definition}[Topological pressure] Let $F:T^1M\to \R$ be a (bounded) continuous  potential. We define the \emph{topological pressure} of $F$ as 
$$P(F)=\sup_{\mu\in \M(g)}\{h_\mu(g)+\int Fd\mu\}.$$
\end{definition}
 A measure $\mu\in \M(g)$ satisfying $P(F)=h_\mu(g)+\int Fd\mu$, is called an \emph{equilibrium state} for the potential $F$. 
For compact dynamical systems there is a well known relation--the variational principle--between the topological pressure of a continuous potential and a weighted version of the topological entropy (see \cite[Chapter 9]{wa}). Similar to the definition of the measure theoretic entropy at infinity (defined in the introduction) we have the following more general definition. 

\begin{definition}[Measure theoretic pressure at infinity]\label{defpresinf} We define the \emph{measure theoretic pressure at infinity} of $F\in C_b(T^1M)$--which we denote by $P_\infty(F)$--by the formula 
$$P_\infty(F)=\sup_{(\mu_n)_n\to 0}\limsup_{n\to\infty} \big( h_{\mu_n}(g)+\int Fd\mu_n\big),$$
where the supremum runs over sequences $(\mu_n)_n$ that converges vaguely to the zero measure. 
\end{definition}


It turns out that the topological pressure of  a H\"older potential has a nice characterization in terms of some critical exponent. Moreover, with H\"older regularity there exists at most one equilibrium state. Before making precise those results we start with some notation. Given two points $x,y\in\widetilde{M}$, we  denote by $[x,y]$ the oriented geodesic segment  starting at $x$ and ending at $y$. For a function $G:T^1\widetilde{M}\to \R$, we  use the notation $\int_x^y G$ to represent the integral of $G$ over the  tangent vectors to the path $[x,y]$ (in direction from $x$ to $y$). Given a potential $F:T^1M\to \R$, we denote by $\widetilde{F}:T^1\widetilde{M}\to \R$ the function $\widetilde{F}=F\circ p$, where $p$ is the canonical projection $p:T^1\widetilde{M}\to T^1M$.  The following definition was introduced in \cite{pps} (a similar definition was used earlier in \cite{cou}).

\begin{definition}\label{cexdef} Let $F$ be a continuous potential and $\widetilde{F}$ its lift to $T^1\widetilde{M}$. Define the \emph{Poincar\'e series associated to $(\Gamma,F)$} based at $z\in \widetilde{M}$ as
$$P(s,F)=\sum_{\gamma\in\Gamma} \exp\bigg(\int_z^{\gamma z}(\widetilde{F}-s)\bigg).$$
 The \emph{critical exponent of $(\Gamma,F)$} is $$\delta^F_\Gamma=\inf\{\text{s } | \text{ }P(s,F)\text{ is finite}\}.$$We say that the pair $(\Gamma,F)$ is of \emph{convergence type} if $P(\delta_\Gamma^F,F)<\infty$, in other words, the Poincar\'e series converges at its critical exponent. Otherwise we say $(\Gamma, F)$ is of \emph{divergence type}. 
\end{definition}

We use the notation $\delta_\Gamma$ to denote the critical exponent of $(\Gamma,0)$, where $0$ is the zero potential. 
 If $F$ is a H\"older potential, then the critical exponent does not depend on the base point $z$. We also remark that if $F$ is bounded, then $\delta_\Gamma^F$ is finite. We next state one of the main results in \cite{pps}, which is a crucial input in this work. 

\begin{theorem}\label{pps} \cite[Theorem 2.3]{pps} Let $F$ be a bounded H\"older potential. Then 
$$P(F)=\delta_\Gamma^F.$$
Moreover, there exists at most one equilibrium state of $F$. 
\end{theorem}
\begin{remark}It follows from the continuity of $P(F)$ and $\delta_\Gamma^F$ under $C^0$ limits of potentials and the density of H\"older potentials among uniformly continuous functions, that Theorem \ref{pps} also holds for uniformly continuous potentials (the same idea is employed in the proof of Lemma \ref{indebase}). In other words, the critical exponent of a uniformly continuous potential is equal to its topological pressure defined through the variational principle. This fact will be frequently used in this paper. 
\end{remark}
If $(\Gamma, F)$ is of convergence type, then $F$ does not have an equilibrium state. If $(\Gamma, F)$ is of divergence type, then it is possible to construct a measure $m_F$ that whenever finite (and normalized to be a probability measure), it is the equilibrium state of the potential $F$. We refer to the measure $m_F$ as the \emph{Gibbs measure associated to}  $F$ (for its construction we refer the reader to \cite{pps}). We remark that Theorem \ref{pps}  was obtained by Otal and Peigne in the case $F=0$ (see \cite{op}). The proof of Theorem \ref{pps} follows very closely the proof of \cite[Theorem 1]{op}. The equality between the topological entropy and the critical exponent of the group was obtained by Manning in the compact case (see \cite{man}).

We will now explain the main results in \cite{sp}, where a characterization for the existence of equilibrium states is provided. Given $\widetilde{U}\subset\widetilde{M}$ we define $\Gamma_{\widetilde{U}}$ as the set of elements in $\Gamma$ such that there exists a geodesic starting at $\widetilde{U}$ and finishing at $\gamma\widetilde{U}$ that only meet $\Gamma\widetilde{U}$ at the beginning and at the end of its trajectory. More precisely we have
$$\Gamma_{\widetilde{U}}=\{\gamma\in \Gamma: \exists y,y'\in \widetilde{U}, [y,\gamma y']\cap g\widetilde{U}\ne \emptyset \Rightarrow \overline{\widetilde{U}}\cap g\overline{\widetilde{U}}\ne \emptyset\text{, or }\gamma\overline{\widetilde{U}}\cap g\overline{\widetilde{U}}\ne \emptyset\}.  $$
Let $\P$ be the set of closed geodesics and $n_U(p)$ the number of times a geodesic $p\in\P$ crosses $U=\pi(\widetilde{U})$. Recall that $\widetilde{\Omega}=p^{-1}(\Omega)$ is the preimage of the non-wandering set under $p:T^1\widetilde{M}\to T^1M$.
\begin{definition} A H\"older potential $F:T^1M\to \R$ is said to be \emph{recurrent} if there exists an open relatively compact subset $U\subset  M$, such that $T^1U\cap \Omega\ne \emptyset$, and $$\sum_{p\in\P}n_U(p)\exp(\int_p F-P(F))=\infty.$$

\end{definition}

\begin{definition} We say that the pair $(\Gamma,\widetilde{F})$ is \emph{positive recurrent} with respect to $\widetilde{U}\subset\widetilde{M}$ if the following properties hold.
\begin{enumerate}
\item $T^1\widetilde{U}$ has non-empty intersection with $\widetilde{\Omega}$. 
\item $F$ is a recurrent potential.
\item There exists $x\in\widetilde{M}$ such that $\sum_{\gamma\in\Gamma_{\widetilde{U}}}d(x,\gamma x)\exp(\int_x^{\gamma x} \widetilde{F}-P(F)),$ is finite. 
\end{enumerate}

\end{definition}

\begin{theorem} \label{sptheo} \cite[Theorem 2]{sp} Let $F$ be a H\"older potential with finite pressure. Then 
\begin{enumerate}
\item If $F$ is recurrent and $(\Gamma,\widetilde{F})$ is positive recurrent with respect to some open relatively compact set $\widetilde{U}\subset \widetilde{M}$, then $m_F$ is finite.
\item If $m_F$ is finite, then $F$ is recurrent and $(\Gamma,\widetilde{F})$ is positive recurrent with respect to any open relatively compact set $\widetilde{U}\subset \widetilde{M}$ meeting the projection  of $\widetilde{\Omega}$ to $\widetilde{M}$. 
\end{enumerate}
\end{theorem}

To introduce the class of strongly positive recurrent potentials we will need some  additional definitions. Let $Q\subset T^1\widetilde{M}$ be a compact subset and $W\subset T^1\widetilde{M}$ an open relatively compact  subset such that $W\cap \widetilde{\Omega}\ne \emptyset$. Let $D\in \N$. Given $n\in\N\cap [2D,\infty)$ we define
\begin{align*}\Gamma^W_Q(n,D)=\{\gamma\in\Gamma:\exists x\in W\cap \widetilde{\Omega}&\text{ such that }g_s(x)\in (\Gamma Q)^c,\text{ for }s\in [D,n-D] \\ &\text{ and }g_{u}(x)\in  \gamma W,\text{ for some }u\in (n-1,n]\}.\end{align*}
Then define $\Gamma_Q^W(D)=\bigcup_{n\ge 2D} \Gamma_Q^W(n,D)$. Observe that if $Q_1\subset Q_2$, then it follows from the definition that $\Gamma_{Q_2}^W(D)\subset \Gamma_{Q_1}^W(D)$.

\begin{definition}[Strongly positive recurrent potentials]\label{sprdef}  Let $W$ and $Q$ be subsets of  $T^1\widetilde{M}$ as above, and $F\in C_b(T^1M)$ a uniformly continuous potential.  We define $\delta^F_\Gamma(Q,W,D)$ as the critical exponent of the Poincar\'e series $$P(F,Q,W,D, s)=\sum_{\gamma\in\Gamma^W_{Q}(D)} \exp(\int_x^{\gamma x} \widetilde{F}-s).$$
Then define $$\delta^F_{\Gamma,\infty}=\inf_{(Q_n)_n}\inf_{n}\sup_D\delta^F_\Gamma(Q_n,W,D),$$
where the infimum runs over sequences of compact sets $(Q_n)_n$ satisfying $Q_n\subset Q_{n+1}$, and $\Omega\subset p(\bigcup_n Q_n)$. The quantity $\delta^F_{\Gamma,\infty}$ is called the \emph{critical exponent at infinity} of the potential $F$. We say that $F$ is \emph{strongly positive recurrent} (SPR for short) if $\delta^F_{\Gamma,\infty}<\delta_\Gamma^F$.
\end{definition}

Implicit in the definition of the critical exponent at infinity of $F$ is the independence on the base point $x\in \widetilde{M}$ and on the open relatively compact set $W\subset T^1\widetilde{M}$. We will explain those facts in Claim 1 and Claim 2.\\

{\bf Claim 1:} $\delta_\Gamma^F(Q,W,D)$ is independent on the base point. \\
If $F$ is H\"older continuous then the critical exponent of the Poincar\'e series of $(\Gamma,F)$ is independent of the base point (see \cite[Lemma 3.3]{pps}; this is a simple application of Lemma \ref{lem:imp}). The same proof applies to the Poincar\'e series $P(F,Q,W,D,s)$. This gives us that  $\delta_\Gamma^F(Q,W,D)$ is well defined and independent of the base point if $F$ is H\"older continuous. In Lemma \ref{indebase} we prove that the same holds for uniformly continuous potentials, that is, for those potentials $\delta_\Gamma^F(Q,W,D)$ is well defined (an analogous result holds for $\delta_\Gamma^F$). 

\begin{lemma}\label{indebase} Let $F$ be a bounded uniformly continuous potential. Then the critical exponent of $P(F, Q, W,D,s)$ is independent of the base point used in Definition \ref{sprdef}. In particular $\delta_\Gamma^F(Q,W,D)$ is well defined. Similarly, $\delta^F_{\Gamma}$ is well defined an independient of the base point. 
\end{lemma}
\begin{proof} To keep track of the base point we define $$P(G,Q,W,D,s;x)=\sum_{\gamma\in\Gamma^W_Q(D)} \exp(\int_x^{\gamma x} \widetilde{G}-s).$$
The critical exponent of $P(G,Q,W,D,s;x)$ is denoted by $\delta_\Gamma^F(Q,W,D;x)$. As mentioned above, if $G$ is H\"older continuous, then the critical exponent of $P(G,Q,W,D,s;x)$ is independent of the base point $x$. It is a standard fact that Lipschitz functions  are dense (in the $C^0$-topology) in the space of uniformly continuous functions of $T^1M$ (we are only using that $T^1M$ is a manifold). In particular we can find a sequence $(F_n)_{n}$ of H\"older potentials converging to $F$ in the $C^0$-topology. There exists a subsequence $(n_k)_k$ such that  $||F_{n_k}-F||_0\le \frac{1}{k}$. Observe that 
\begin{align*}\sum_{\gamma\in\Gamma^W_Q(D)}  \exp\bigg(\int_x^{\gamma x} \widetilde{F_{n_k}}-\bigg(s+\frac{1}{k}\bigg)\bigg)&\le \sum_{\gamma\in\Gamma^W_Q(D)}  \exp\bigg(\int_x^{\gamma x} \widetilde{F}-s\bigg)\\ &\le \sum_{\gamma\in\Gamma^W_Q(D)}  \exp\bigg(\int_x^{\gamma x} \widetilde{F_{n_k}}-\bigg(s-\frac{1}{k}\bigg)\bigg).
\end{align*}
From this we obtain the inequality 
\begin{align}\label{sh}\delta_{\Gamma}^{F_{n_k}}(Q,W,D;x)- \frac{2}{k}\le \delta_{\Gamma}^F(Q,W,D;x)\le \delta_{\Gamma}^{F_{n_k}}(Q,W,D;x)+ \frac{2}{k}.\end{align}
We conclude that $\lim_{k\to\infty} \delta_\Gamma^{F_{n_k}}(Q,W,D;x)=\delta_{\Gamma}^F(Q,W,D;x)$. Since $F_{n_k}$ is H\"older continuous we know that $\delta_{\Gamma}^{F_{n_k}}(Q,W,D;x)=\delta_{\Gamma}^{F_{n_k}}(Q,W,D;y)$, for every $(x,y)\in M\times M$. We conclude that $$\delta_{\Gamma}^F(Q,W,D;y)=\lim_{k\to\infty} \delta_\Gamma^{F_{n_k}}(Q,W,D;y)=\lim_{k\to\infty} \delta_\Gamma^{F_{n_k}}(Q,W,D;x)=\delta_{\Gamma}^F(Q,W,D;x),$$ which is what we wanted to prove. The calculation for the usual critical exponent is analogous.

\end{proof}

Before moving to Claim 2 we will prove a useful formula for $\delta_\Gamma^F(Q,W,D)$ in case $F$ is H\"older continuous. We will need the following simple, but important result (see \cite[Lemma 3.2]{pps}). 

\begin{lemma}\label{lem:imp} Let $F$ be an H\"older potential in $T^1M$. Then for every $x,y,z,w\in \widetilde{M}$ we have that 
$$|\int_x^y \widetilde{F}-\int_z^w \widetilde{F}|\le C(d(x,z), d(y,w), F, M),$$
for some function $C$ that only depends on $d(x,z)$, $d(y,w)$, the H\"older constants of the potential $F$ and bounds on the  sectional curvature of $\widetilde{M}$. 
\end{lemma}

As before, we are given $Q$, $W$ and $D$. Choose a reference point $z\in \widetilde{M}$, and let $E=\sup_{x\in \pi(W)}d(z,x)$. Let $\gamma\in \Gamma^W_Q(n,D)$. By definition there exists a point $x^n_\gamma \in W\cap \widetilde{\Omega}$ such that $g_k(x^n_\gamma)\in (\Gamma Q)^c$ for $k\in [D,n-D]$ and $g_u(x^n_\gamma)\in \gamma W$, for some $u\in [n-1,n]$. Observe that $d(z,\pi(x^n_\gamma))$ and $d(\gamma z, \pi(g_n(x^n_\gamma)))$ are both bounded above by $E$. In particular, by Lemma \ref{lem:imp} there exist $N_1>0$ independent of $\gamma$ (but depending on $s$, since $N_1$ depend on the potential we are using) such that 
$$|\int_z^{\gamma z}(\widetilde{F}-s)-\int_{\pi(x^n_\gamma)}^{\pi(g_n(x^n_\gamma))}(\widetilde{F}-s)|\le N_1.$$
From the triangle inequality we have that $d(z,\gamma z)\in [n-2E-1, n+2E]$. In particular if $\tau\in \bigcap_{i=1}^k \Gamma^W_Q(n_i,D)$, then $d(z,\tau z)\in \bigcap_{i=1}^k[n_i-2E-1, n_i+2E]$. We conclude that $\tau\in \Gamma$ can belong to at most $4(E+1)$ different sets from $\{\Gamma^W_Q(m,D)\}_{m\ge 2}$. Combining all this we get

\begin{align*} 4(E+1) \sum_{\gamma\in \Gamma^W_Q(D)}\exp\big(\int_{z}^{\gamma z}(\widetilde{F}-s)\big)&\ge \sum_{n\ge 2D}\sum_{\gamma\in \Gamma_{Q}^W(n,D)}\exp\big(\int_z^{\gamma z}\widetilde{F}-s\big)\\
&\ge  \sum_{\gamma\in \Gamma^W_Q(D)}\exp\big(\int_{z}^{\gamma z}(\widetilde{F}-s)\big).
\end{align*}
In particular the critical exponent of the Poincar\'e series of $P(F,Q,W,D,s)$ 
is equal to the critical exponent of $$\sum_{n\ge 2D}\sum_{\gamma\in \Gamma^W_{Q}(n,D)}\exp\big(\int_z^{\gamma z}\widetilde{F}-s\big).$$
It follows from $d(z, \gamma z)\in [n-2E-1,n+2E]$ whenever $\gamma\in \Gamma^W_Q(n,D)$ that for $s\ge 0$ we have (the reversed inequality holds if $s<0$)
\begin{align*}\sum_{n\ge 2D}e^{-s(n+2E)}\sum_{\gamma\in \Gamma^W_Q(n,D)}\exp\big(\int_z^{\gamma z}\widetilde{F}\big)&\le \sum_{n\ge 2D}\sum_{\gamma\in \Gamma^W_{Q}(n,D)}\exp\big(\int_z^{\gamma z}\widetilde{F}-s\big)\\
&\le \sum_{n\ge 2D}e^{-s(n-2E-1)}\sum_{\gamma\in \Gamma_Q^W(n,D)}\exp\big(\int_z^{\gamma z}\widetilde{F}\big).\end{align*}
We conclude that critical exponent of the Poincar\'e series $P(F,Q,W,D,s)$ is given by 
\begin{align}\label{xxxxy} \delta^F_\Gamma(Q,W,D)=\limsup_{n\to\infty}\frac{1}{n}\log \sum_{\gamma\in \Gamma^W_Q(n,D)}\exp\big(\int_z^{\gamma z}\widetilde{F}\big),\end{align}
provided that $F$ is H\"older continuous.  \\

{\bf Claim 2:} The quantity $$\inf_{(Q_n)_n}\inf_{n}\sup_D\delta^F_\Gamma(Q_n,W,D),$$ is independent of $W$. 
\\
We will first assume that $F$ is H\"older continuous, and prove that 
\begin{align}\label{ppp}\sup_D \delta_\Gamma^F(Q_0,W_0,D)\le \sup_D\delta_\Gamma^F(Q_1,W_1,D),\end{align}
for every pair for open relatively compact sets $W_0$ and $W_1$ in $T^1\widetilde{M}$ intersecting $\widetilde{\Omega}$, and $Q_1\subset Q_0$ such that $d(Q_1,T^1\widetilde{M}\setminus Q_0)>0$.

We will sketch the idea--the same argument will be detailed in the proof of Lemma \ref{lem:com}. It is enough to prove that $\delta_\Gamma^F(Q_0,W_0,D)\le \delta_\Gamma^F(Q_1,W_1,D_1)$, for some $D_1$ sufficiently large. We emphasize that the main ingredients in this proof are the closing lemma, the local product structure and the transitivity of the geodesic flow on $\Omega$ (for precise definitions see Section \ref{4}).


By the transitivity of the geodesic flow (on $\Omega$), for every $\delta>0$, there exists $L(\delta,p (W_0),p(W_1))>0$ such that every pair of points $x_0\in p(W_0)\cap \Omega$, $x_1\in p(W_1)\cap \Omega$ can be connected with a geodesic of length at most $L$ up to $\delta$-perturbations of $x_0$ and $x_1$ (we use points nearby $x_0$ and $x_1$). We choose $L$ so that this works for paths going from $p(W_0)$ to $p(W_1)$, and from $p(W_1)$ to $p(W_0)$. Choose a closed ball $B$ which contains a 1-neighboorhood of $p(W_0)$, $p(W_1)$ and of each connecting geodesic of length at most $L$ just described. Set $M=\max_{x\in T^1B}|F(x)|$. 

 Let $\gamma \in \Gamma_{Q_0}^{W_0}(n,D)$ and $x:=x_\gamma\in W_0\cap \widetilde{\Omega}$ such that $g_s(x)\in (\Gamma Q_0)^c$, for $s\in [D,n-D]$ and $g_u(x)\in \gamma W$, for some $u\in[n-1,n]$.  Fix a point $\hat{x}\in W_1\cap \widetilde{\Omega}$. Now consider the following three geodesic segments: the segment that (approximately) goes from $p(\hat{x})$ to $p(x)$, the geodesic flow of $p(x)$ up to time $u$, and the segment that (approximately) goes from $g_u(p(x))$ to $p(\hat{x})$. Taking $\delta$ sufficiently small and $n$ sufficiently large we can use the closing lemma to find a closed geodesic that $\e$-shadows these three geodesics segments (for us $\e>0$ will be very small).  Take a point $x'\in W_1$ which is $\e$-close to $\hat{x}$ and that is tangent to a lift of the closed geodesic just constructed. Let $H(\gamma)\in \Gamma$ be the hyperbolic isometry with axis tangent to $x'$, and $l(H(\gamma))$ its translation length. If $\e$ is sufficiently small  (that is, $\delta$ sufficiently small and $n$ sufficiently large), then the assignment $\gamma\in \Gamma_{Q_0}^{W_0}(n,D)\mapsto H(\gamma)\in \Gamma$ is an injection. We will moreover assume that $\e<d(Q_1,T^1\widetilde{M}\setminus Q_0)$. We remark that the size of $\delta$ determines $L$, which is now fixed. Observe that the point $x'\in W_1\cap \widetilde{\Omega}$ satisfies that $g_s(x')\in (\Gamma Q_1)^c$, for $s\in [L+D+1, l(H(\gamma))-(L+D+1)]$ and $g_{l(H(\gamma))}(x')\in H(\gamma) W_1$. By construction we know that $l(H(\gamma))\in [n-1, n+2L]$. In particular we have that 
\begin{align}\label{xyz} H(\gamma)\in \bigcup_{s=0}^{2L+1}\Gamma_{Q_1}^{W_1}(n-1+s,L+D+1).\end{align} 
Using Lemma \ref{lem:imp} we can compare $\int_{x_{\gamma_1}}^{\gamma_1 x_{\gamma_1}}\widetilde{F}$ with $\int_x^{\gamma_1 x}\widetilde{F}$, for every $x_{\gamma_1}$ (where $\gamma_1\in \Gamma_{Q_0}^{W_0}(n,D)$). By the construction of $H(\gamma)$--in terms of shadowing--and Lemma \ref{lem:imp} we get 
\begin{align}\label{ti}-2LM-C+\int_{x}^{\gamma x}\widetilde{F} \le \int_{\hat{x}}^{H(\gamma)\hat{x}}\widetilde{F},\end{align}
for some $C>0$ independent of $\gamma\in \Gamma_{Q_0}^{W_0}(n,D)$. Using (\ref{ti}), (\ref{xyz}) and (\ref{xxxxy}) we get that 
$\delta_\Gamma^F(Q_0,W_0,D)\le \delta_\Gamma^F(Q_1,W_1,L+D+1)$, as desired.

Recall that if $Q'\subset Q$, then $\Gamma_{Q}^W(D)\subset \Gamma_{Q'}^W(D)$. In particular we have
\begin{align}\label{ppp2}\delta_\Gamma^F(Q,W,D)\le \delta_\Gamma^F(Q',W,D).\end{align} We construct $(Q_n')_n$ inductively such that $Q_n\subset Q_n'$, $Q_n'\subset Q_{n+1}'$ and $d(Q_n',\widetilde{M}\setminus Q_{n+1}')>0$. Combining (\ref{ppp}) and (\ref{ppp2}) we get  $$\sup_D \delta_\Gamma^F(Q_{n+1}',W_0,D)\le \sup_D\delta_\Gamma^F(Q_n',W_1,D)\le \sup_D\delta_\Gamma^F(Q_n,W_1,D),$$
therefore $\inf_n \sup_D \delta_\Gamma^F(Q_{n+1}',W_0,D)\le  \inf_n\sup_D\delta_\Gamma^F(Q_n,W_1,D).$ Analogously we get $\inf_n \sup_D \delta_\Gamma^F(Q_{n+1}',W_1,D)\le  \inf_n\sup_D\delta_\Gamma^F(Q_n,W_0,D).$ This implies the independence on $W$ of the critical exponent at infinity, at least when $F$ is H\"older continuous. 

For $F$ uniformly continuous we use an approximation $(F_k)_k$ by H\"older potentials such that $||F-F_k||\le \frac{1}{k}$. As proved in (\ref{sh}) we have the bound  $$\delta_\Gamma^{F_k}(Q,W,D)-\frac{2}{k}\le \delta_\Gamma^F(Q,W,D)\le \delta_\Gamma^{F_k}(Q,W,D)+\frac{2}{k}.$$
It follows from this comparison, that the result for H\"older potentials propagates to any uniformly continuous potential. This finishes the proof of Claim 2. We conclude that $\delta_{\Gamma,\infty}^F$ is well defined for uniformly continous potentials. \\




In the geometrically finite case it is easy to verify that $$\delta^F_{\Gamma,\infty}=\sup_\P \delta_\P^F,$$
where the supremum runs over the parabolic subgroups of $\pi_1(M)$. This follows from the structure of the ends of $\Omega$ and the convexity of the horoballs. In particular a potential $F$ is SPR if and only if $\sup_\P \delta_\P^F<P(F)$. 

We emphasize that  SPR potentials are expected to have a thermodynamic formalism very similar to the existing one for potentials on compact manifolds. In Section \ref{6} we will prove the main properties of the family of SPR potentials within certain classes of potentials that do not oscillate at infinity (we will introduce them in Section \ref{funspaces}). 

We finish this section by recalling a remarkable property of the Gibbs measure associated to $F$. We use the notation $B^p_n(v,r)$ to denote the projection to $T^1M$ of the dynamical ball $B_n(\widetilde{v},r)$ in $T^1\widetilde{M}$, where $\widetilde{v}$ is a lift of $v$ (for a precise definition see Definition \ref{defp}). We say that $B^p_n(v,r)$ is a $p(n,r)$-dynamical ball. 

\begin{definition}[Gibbs property]\label{gibbsprop} Let $m$ be an invariant measure on $T^1M$. We say that $m$ verifies the \emph{Gibbs property for the potential} $F:T^1M\to \R$ if for every compact set $K\subset T^1M$ and $r>0$ there exists a constant $C_{K,r}\geq 1$ such that for every $v\in K$ and every $n\geq 1$ such that $g_n(v)\in K$, we have
$$C_{K,r}^{-1}\leq \frac{m(B^p_n(v,r))}{e^{\int_0^n F(g_t(v))-P(F)dt}}\leq C_{K,r}.$$
\end{definition}

This definition resembles the usual Gibbs property in symbolic dynamics, but it is not identical (the compact subset plays a role here). For sub-shifts of finite type, every equilibrium state of a H\"older potential satisfies the Gibbs property \cite{bob}. For countable Markov shifts Gibbs measures exist only if the shift space satisfies the BIP property (see \cite{sa5}). It is proven in \cite[Section 3.8]{pps} that the Gibbs measures of H\"older potentials satisfy the Gibbs property. This fact will be important in order to prove our modification of the simplified entropy inequality (see Theorem \ref{indeprad}).

\subsection{Relevant spaces of potentials}\label{funspaces}
In this paper we will study the thermodynamic formalism of certain classes of potentials.  

\begin{definition}\label{limD} We say that a potential $F$ \emph{converges to $D$ at infinity} if for each $\epsilon>0$, there exists a compact subset $K\subset T^1M$ such that $$\sup_{x\in K^c}|F(x)-D|<\epsilon.$$ The space of continuous potentials converging to $D$ at infinity is denoted by $C(T^1M,D)$. The set $C_0(T^1M):=C(T^1M,0)$ is the space of potentials that \emph{vanish at infinity}.
\end{definition}
We use the notation $C(T^1M,D)$ to avoid any confusion with the space of bounded continuous potentials $C_b(T^1M)$. Define $$\cH=\bigcup_{D\in \R}C(T^1M,D).$$ We remark that potentials in $\cH$ are uniformly continous. 

\begin{lemma} \label{c0top}Let $F\in C_0(T^1M)$. Then the map $$\mu\mapsto \int Fd\mu,$$ is continuous in $\M_{\le 1}(g)$. 
\end{lemma}
\begin{proof} Suppose we have a sequence $(\mu_n)_n\subset\M_{\le 1}(g)$ converging in the vague topology to $\mu$. Fix $\epsilon>0$, and let $K=K(\epsilon)$ be a compact subset of $T^1M$ such that $\mu(\partial K)=0$, and $\sup_{x\in K^c} |F(x)|<\epsilon.$ Since $\mu(\partial K)=0$ we know that $\lim_{n\to\infty}\int_K Fd\mu_n=\int_K Fd\mu$. Observe that 
\begin{align*}|\int Fd\mu-\int Fd\mu_n|\le & |\int_K Fd\mu-\int_K Fd\mu_n|+|\int_{K^c} Fd\mu-\int_{K^c} Fd\mu_n|\\
 \le & |\int_K Fd\mu-\int_K Fd\mu_n|+2\e.
\end{align*}
Finally 
$$\limsup_{n\to \infty}|\int Fd\mu-\int Fd\mu_n|\le \limsup_{n\to\infty}|\int_K Fd\mu-\int_K Fd\mu_n|+2\epsilon = 2\epsilon.$$
Since $\epsilon>0$ was arbitrary we are done.
\end{proof}

As a consequence of Lemma \ref{c0top} be obtain the following result. 
\begin{lemma}\label{D}Let $(\mu_n)_n$ be a sequence of invariant probability measures converging vaguely to $\mu$. Then for every $F\in C(T^1M,D)$ we have $$\lim_{n\to\infty}\int Fd\mu_n=\int Fd\mu+(1-|\mu|)D.$$
\end{lemma}
\begin{proof} Define $G=F-D$. It is clear that $G\in C_0(T^1M)$. By Lemma \ref{c0top} we get 
$$\lim_{n\to\infty} \int Gd\mu_n=\int Gd\mu.$$
This implies that $\lim_{n\to\infty} \int Fd\mu_n=\int Fd\mu+(1-|\mu|)D$.
\end{proof}

Recall that $\Omega$ is the non-wandering set of the geodesic flow.

\begin{definition}\label{Edef} For a potential $F\in C_b(T^1M)$ we define $$c(F)=\inf_{n\in \N} \sup_{x\in K_n^c}F(x),$$ 
where $(K_n)_n$ is a compact exhaustion of $\Omega$.
\end{definition}
It is easy to verify that $c(F)$ is independent of the choice of compact exhaustion. Our next result will be frequently used in Section  \ref{6}.

\begin{lemma}\label{E}Let $(\mu_n)_n$ be a sequence of invariant probability measures converging vaguely to $\mu$. Then for every $F\in C_b(T^1M)$ we have $$\limsup_{n\to\infty}\int Fd\mu_n\le \int Fd\mu+(1-|\mu|)c(F).$$
\end{lemma}
\begin{proof}  Let $(K_m)_m$ be a compact exhaustion of $\Omega$ such that the following hold:
\begin{enumerate}
\item\label{1tt} $\sup_{x\in {\Omega\setminus K_m}}F(x)<c(F)+\frac{1}{m}$.
\item\label{2tt} $\mu(\partial K_m)=0$.
\end{enumerate}
It follows from (\ref{2tt}) that  $$\lim_{n\to\infty}\int_{K_m} Fd\mu_n=\int_{K_m} Fd\mu.$$   Similarly $\lim_{n\to\infty}\mu_n(K_m)=\mu(K_m)$, and therefore $\lim_{n\to\infty}\mu_n(\Omega\setminus K_m)=1-\mu(K_m)$. Observe that (\ref{1tt}) gives us $\int_{\Omega\setminus K_m }Fd\mu_n\le (c(F)+\frac{1}{m})\mu_n(\Omega\setminus K_m)$, and that 
\begin{align*}\bigg( \int Fd\mu_n-\int_{K_m} Fd\mu - & (1-\mu(K_m))c(F)\bigg)\\
=&\bigg(\int_{K_m} Fd\mu_n-\int_{K_m} Fd\mu\bigg)+\bigg(\int_{\Omega\setminus K_m}Fd\mu_n-c(F)\mu_n(\Omega\setminus K_m)\bigg)\\
&+(c(F)\mu_n(\Omega\setminus K_m)-c(F)(1-\mu(K_m))).
\end{align*}
Taking limsup as $n$ goes to infinity we obtain that 
$$\limsup_{n\to\infty}\int Fd\mu_n-\bigg(\int_{K_m}Fd\mu+c(F)(1-\mu(K_m))\bigg)\le \frac{1}{m}.$$
Sending $m$ to infinity proves the lemma. 
\end{proof}

\section{Entropy density and simplified entropy formula} \label{4}
 In this section we will prove that the geodesic flow on a pinched negatively curved manifold satisfies a mild modification of the  simplified entropy inequality and that its ergodic measures are weak entropy dense. These two properties will be used to prove the upper semicontinuity of the entropy map.   

\subsection{Weak entropy density}\label{entdensec}
In this section we will check that the proof of \cite[Theorem B]{ekw} extends to the non-compact case for the geodesic flow on a negatively curved manifold.  

Let $(X,d)$ be a metric space and $(\phi_t)_{t\in \R}$ a continuous flow on $X$. We will need the following definitions. 

\begin{definition}[Closing lemma] We say that the flow $(\phi_t)_{t}$ satisfies the closing lemma if for all $x\in X$, there exists a neighborhood $W_x$ of $x$ such that the following holds. Given $\e>0$, there exists $\delta$ and $t_0$ such that for all $y\in W_x$ and $t\ge t_0$, if $d(y,\phi_t y)<\delta$ and $\phi_t y\in W_x$, then there exists $y'$ and $s>0$ such that $|t-s|<\e$, $\phi_s y'=y'$ and $d(\phi_h y,\phi_h y')<\e$, for $h\in (0,\min\{t,s\})$. 
\end{definition}

Define the sets 
$$W^{ss}_\e(x)=\{y\in X: d(\phi_t(x),\phi_t(y))\le \e,\text{ for all }t\ge 0\},$$
$$W^{su}_\e(x)=\{y\in X: d(\phi_t(x),\phi_t(y))\le \e, \text{ for all }t\le 0\}.$$

\begin{definition}[Local product structure] We say that the flow $(\phi_t)_{t}$  admits a local product structure if for all $x\in X$, there exists a neighborhood $V_x$ of $x$ such that the following holds. Given $\e>0$, there exists $\delta>0$ such that for all $y,z\in V_x$ satisfying $d(y,z)<\delta$, there exists a point $w=\langle y,z\rangle \in X$ and a real number $t\in (-\e,\e)$ so that  $\langle y,z\rangle\in W_\e^{su}(\phi_t(x))\cap W^{ss}_\e(y).$ 

\end{definition}

The space of flow invariant probability measures is denoted by $\M(X,\phi)$. In this context we say that $\mu$ is ergodic if it is an ergodic flow invariant probability measure. The space of ergodic measures is denoted by $\M_e(X,\phi)$. 

\begin{remark}\label{rementropydense} It is a well known fact (see for instance \cite{bal}) that the geodesic flow of a pinched negatively curved manifold restricted to its non-wandering set $\Omega$ is transitive, satisfies the closing lemma and admits local product structure. These properties are important for us because of the following result.
\end{remark}


\begin{proposition}\label{entropydense} Let $(X,d)$ be a metric space. Assume that closed balls on $X$ are compact.  Let $(\phi_t)_{t}$ be a continuous flow which is transitive, admits local product structure and satisfies the closing lemma. Then for every measure $\mu\in \M(X,\phi)$ and $\epsilon>0$, there exists an ergodic measure $\mu_e\in \M_e(X,\phi)$ arbitrarily close to $\mu$ (in the weak* topology) such that $ h_{\mu_e}(\phi_1)>h_\mu(\phi_1)-\epsilon$. We can moreover assume that $\supp \mu_e$ is compact.
\end{proposition}

\begin{proof}  
As in the proof of the entropy density in the compact case we start with the following general fact.
\begin{lemma}\cite[Proposition 6.1]{ekw} \label{entropyforergodic} Let $(X,d)$ be a metric space and $T$ a continuous transformation. Given an ergodic measure $\mu$, $\alpha>0$, $\beta>0$ and $f_1,...,f_l\in C_b(X)$, there exists $n_0$ and $\gamma>0$ such that for all $n\ge n_0$ there exists a $(n,\gamma)$-separated set $S\subset X$ such that 
\begin{enumerate}
\item $|S|\ge \exp(n(h_\mu(T)-\alpha))$.
\item\label{iix} $|\frac{1}{n}\sum_{k=0}^{n-1} f_j(T^k x)-\int f_j d\mu|<\beta$, for all $x\in S$ and $j\in\{1,...,l\}$.
\end{enumerate}
Let $K\subset X$ be a measurable set satisfying  $\mu(K)>3/4$. Then we can moreover assume that $S\subset K\cap T^{-n}K$. The constant $\gamma$ does not depend on $K$, or $n$.
\end{lemma} 
We only need to justify the last part of the proposition since  points (1) and (2) are taken without modification from \cite{ekw}. We follow the notation in the proof of \cite[Proposition 6.1]{ekw} and  modify the definition of $F_n$ by the formula
 $$F_n=E_n\cap K\cap T^{-n}K\cap \{x\in X: \frac{1}{n}\sum_{k=0}^{n-1}\chi_V(T^k x)\le 2\delta\},$$
where $E_n=\{x\in X: |\frac{1}{n}\sum_{k=0}^{n-1}f_j(T^kx)-\int f_jd\mu|<\beta,\forall j\in \{1,...,l\}\}$. Since $\mu(K\cap T^{-n}K)>\frac{1}{2}$, and the measure of the other two sets involved in the definition of $F_n$, by Birkhoff ergodic theorem, tends to 1 as $n$ goes to infinity, we conclude that for $n$ sufficiently large we have $\mu(F_n)>\frac{1}{2}$. With this definition of $F_n$ the rest of the proof follows without modification the proof of \cite[Proposition 6.1]{ekw}. \\

We want to prove that given $\mu\in \M(X,\phi)$, $\epsilon>0$, $\eta>0$, and $f_1,...,f_l\in C_c(X)$, there exists an ergodic measure $\mu_e\in V(f_1,...,f_l;\mu,\epsilon)$, where $$ V(f_1,...,f_l;\mu,\epsilon)=\{\nu\in \M(X,\phi): |\int f_i d\nu-\int f_id\mu|\le \epsilon,\forall i\in\{1,...,l\}\},$$
and such that $h_{\mu_e}(\phi_1)>h_\mu(\phi_1)-\eta.$  As in the proof of \cite[Theorem B]{ekw} we can reduce the problem to the case $\mu=\frac{1}{N}\sum_{k=1}^N \mu_k$, where $\{\mu_k\}_{k=1}^N$ is a collection of ergodic measures. It is convenient to define $M=\max_{i\in \{1,...,l\}} ||f_i||_0$. 

In order to use Lemma \ref{entropyforergodic} we choose a time $t_0$ for which $\mu_1,...,\mu_k$ are ergodic with respect to $\phi_{t_0}$ (for instance see \cite[Lemma 7]{op}), and use $T=\phi_{t_0}$. In this situation the bound (\ref{iix}) can be replaced by $|\frac{1}{n}\int_{0}^{n} f_j(\phi_t x)dt-\int f_j d\mu_i|<\beta$ (for $x$ in the appropiate set $S_i$). This can be done by changing sums into integrals in the definition of $E_n$ (see paragraph above); the measure of $E_n$ is still close to one by the ergodic theorem for flows. We can take $t_0$ arbitrarily close to $1$, but for simplicity we will pretend that $t_0$ is actually equal to $1$ (otherwise we will need to factor out $t_0$ in many of our computations). 

Fix  a compact set $K$ such that $\mu_i(K)>3/4$, $\forall i\in\{1,...,l\}$. By the uniform continuity of the functions $(f_i)_i$, there exists $\epsilon_0>0$ such that if $d(x,y)<\epsilon_0$, then $|f_i(x)-f_i(y)|<\frac{\epsilon}{4}$. Using Lemma \ref{entropyforergodic} (and its modification for a flow) we can find $n_0$ and $\gamma>0$ such that for every $n\ge n_0$, there exist an $(n,\gamma)$-separated set $S_i=S_i(n,\e,\eta)\subset K\cap \phi_{-n}K$ satisfying
that 
\begin{enumerate}
\item\label{cfea} $|S_i|\ge \exp(n(h_{\mu_i}(\phi_1)-\eta/2))$.
\item\label{clinda} $|\frac{1}{n}\int_{0}^{n} f_j(\phi_t x)dt-\int f_j d\mu_i|<\epsilon/4$, for all $x\in S_i$ and $j\in\{1,...,l\}$.
\end{enumerate}
We will moreover assume that $\epsilon_0$ satisfy $\gamma>\epsilon_0/4$. 

By the definition of the  local product structure and the compactness of $K$  
 there exists $\delta_0=\delta_0(\epsilon_0,K)>0$ such that if $x,y\in K$ satisfy $d(x,y)<\delta_0$, then there exists a point $\langle x,y\rangle$ such that $\langle x,y\rangle \in W_{\epsilon_0}^{su}(\phi_t(x))\cap W_{\epsilon_0}^{ss}(y)$ and $|t|<\epsilon_0$. 

 By the transitivity of the flow and the compactness of $K$, 
there exists a constant $R=R(\delta_0,K)$ such that for every $x,y\in K$, there exists $z\in X$ such that $d(x,z)<\delta_0$, and $d(\phi_p z,y)<\delta_0$ for some $p\in [0,R]$. In particular if $(x,y)\in \big(\bigcup_{k=1}^N S_k\big)^2$, there exists $z=z(x,y)\in X$ such that $d(\phi_n x,z)<\delta_0$, and $d(\phi_p z,y)<\delta_0$, for some $p=p(x,y)\in [0,R]$. By choosing $n$ sufficiently large we can and will assume $R(\delta_0,K)/n$ is sufficiently small.  

Choose $\bar{x}=(x_1,x_2,...,x_{MN})\in (\prod_{i=1}^N S_i)^M$. As in the proof of \cite[Proposition 3.2]{cs}, using the closing lemma and the local product structure, we can  construct a periodic orbit that $\epsilon_0$-shadows the broken orbit 
$$W=O_0^n(x_1)\cup O_0^{p(x_1,x_2)}z(x_1,x_2)\cup O_0^n(x_2)\cup ...\cup O_0^{n}(x_n)\cup O_0^{p(x_{MN},x_1)}z(x_{MN},x_1),$$
where the notation $O_a^b(x)$ represents the piece of orbit $(\phi_t(x))_{t\in [a,b]}$. 
%
We think of $W$ as a parametrized map which represents the sequence of segments above. We denote the periodic orbit shadowing $W$ by $w(\bar{x})=w(x_1,...,x_{MN})$. The period of $w(\bar{x})$ is approximately the domain of $W$, it belongs to $[nNM, (n+R+1)NM]$.

 Define $S(n)=\bigcup_{M\ge 1}(\prod_{i=1}^N S_i)^M$ (here we emphasize that the sets $(S_i)_i$ depend on $n$).  As before, to each element $\bar{x}\in S(n)$ we associate a periodic orbit $w(\bar{x})$. Fix some reference point $q\in K$. There exists a constant $L=L(n,R(\delta_0,K))$ such that $ w(\bar{x})\subset B(q,L)$, for every $\bar{x}\in S(n)$. In particular the set $$\Psi(n)=\bigcup_{\bar{x}\in S(n)}w(\bar{x}),$$
is relatively compact in $X$ and $\phi$-invariant. This implies that $\Psi_0(n)=\overline{\Psi(n)}$ is compact and $\phi$-invariant. Define 
$$A_n=\{x\in X: |\frac{1}{nN}\int_{0}^{nN} f_i(\phi_{t}x)dt-\int f_i d\mu|\le\epsilon, \forall i\in\{1,...,l\}\}.$$
It follows from the choices of $\e_0$, $\delta_0$, and the inequality (\ref{clinda}) that for $n$ large enough (in comparison with $NMR$) we have that $\Psi(n)\subset A_n$. We remark that the constants $\e_0$, $\delta_0$, $N$, $M$ and $R$ do not depend on $n$. Since $A_n$ is closed we automatically get that $\Psi_0(n)\subset A_n$. 

Since $\Psi_0(n)$ is $\phi$-invariant we know that $x\in \Psi_0(n)$ implies that  $\phi_s(x)\in A_n$, for every $s\in \R$. Let $\nu$ be an ergodic measure supported in $\Psi_0$ and $x\in \Psi_0$ a point such that 
\begin{align}\label{birk}\lim_{m\to\infty}\frac{1}{m}\int_0^m f_i(\phi_t x)dt=\int f_i d\nu,\end{align}
for every $i\in \{1,...,l\}$. Observe that $$|\frac{1}{nN}\int_{0}^{nN} f_i(\phi_{t+nNk}x)dt-\int f_i d\mu|\le \epsilon,$$ for every $k\in \N$ and $i\in \{1,...,l\}$. Combining this with (\ref{birk}) we get that $|\int fd_i\mu-\int f_i\nu|\le \e$, for every $\i \in \{1,...,l\}$. This implies that $\nu\in V(f_1,...,f_l;\mu,\e)$. 
We conclude that every ergodic measure supported in $\Psi_0$ belongs to  $V(f_1,...,f_l;\mu,\epsilon)$.

 Recall that $\gamma>\epsilon_0/4$. By construction if $\bar{x},\bar{y}\in (\prod_{i=1}^N S_i)^M$, and $\bar{x}\ne \bar{y}$, then $$d_{MN(n+R)}(w(\bar{x}),w(\bar{y}))>\epsilon_0/2.$$ In other words $\Psi_0$ contains a $(MN(n+R),\epsilon_0/2)$-separated set of cardinality $$\exp\bigg(nNM\bigg(\frac{1}{N}\sum_{k=1}^N h_{\mu_k}(\phi_1)-\frac{\eta}{2}\bigg)\bigg)=\exp\bigg(nNM\bigg(h_\mu(\phi_1)-\frac{\eta}{2}\bigg)\bigg).$$
Here we used inequality (\ref{cfea}). Then $$h_{top}(\Psi_0)\ge \limsup_{M\to \infty} \dfrac{nNM(h_\mu(\phi_1)-\frac{\eta}{2})}{MN(n+R)}.$$
As mentioned earlier, we assumed that $R/n$ is sufficiently small, in particular we require that the right hand side in the last inequality is strictly bigger than $h_\mu(\phi_1)-\eta$. Finally, take an ergodic measure $\mu_e$ supported in $\Psi_0$ with entropy at least $h_\mu(\phi_1)-\eta$ (using the variational principle for flows). Since we had already proved that  $\mu_e\in V(f_1,...,f_l;\mu,\epsilon)$,  this finishes the proof. 
\end{proof}

Combining Remark \ref{rementropydense} and Proposition \ref{entropydense} we obtain the main result of this section. Here we are using that $\Omega$ is the support of all invariant measures on $\M(g)$. 

\begin{theorem}[Weak entropy density of the geodesic flow]\label{wed}  Let $M$ be a pinched negatively curved manifold. Then the ergodic measures are weak entropy dense in the space of invariant probability measures. 
\end{theorem}

Let $A\subset T^1M$ be a closed invariant subset of the geodesic flow. We define the topological pressure of $F$ restricted to $A$ as $$P_A(F)=\sup_\mu\{h_\mu(g)+\int Fd\mu\},$$
where the supremum runs over the invariant probability measures supported on $A$. As a corollary of Theorem \ref{wed} we obtain that the topological pressure can be approximated by compact subsets. This result was obtained in \cite[Lemma 6.7]{pps} for H\"older potentials. 

\begin{theorem}\label{aproxc0}  Let $F\in C_b(T^1M)$, then $$P(F)=\sup_K P_K(F),$$ where the supremum runs over the set of  compact invariant subsets.
\end{theorem}
\begin{proof}
 Fix $\epsilon>0$ and choose $\mu_\e \in \M(g)$ such that $$h_{\mu_\e}(g)+\int Fd\mu_\e>P(F)-\epsilon.$$ Using Proposition \ref{entropydense} we can approximate $\mu_\e$ with an invariant probability measure of compact support $\nu_\e$ satisfying $h_{\nu_e}(g)>h_{\mu_\e}(g)-\epsilon$, and $\int Fd\nu_\e>\int Fd\mu_\e-\e$. This implies that 
$$h_{\nu_\e}(g)+\int Fd\nu_\e>P(F)-3\epsilon,$$
where $\nu_\e$ is compactly supported. In particular $\sup_K P_K(F)\ge h_{\nu_\e}(g)+\int Fd\nu_\e$. It is clear that $P(F)\ge \sup_K P_K(F)$. Therefore $$P(F)\ge \sup_K P_K(F)\ge P(F)-3\e.$$ Since $\e>0$ was arbitrary we conclude that $P(F)=\sup_{K} P_K(F)$.
\end{proof}

\subsection{Simplified entropy inequality for geodesic flows}\label{ent4}
We will now proceed to prove that the geodesic flow on a pinched negatively curved manifold satisfies a simplified entropy inequality after a mild modification in the definition of the dynamical balls.  As before $p:T^1\widetilde{M}\to T^1M$ is the canonical projection.

\begin{definition} \label{defp} Let $y\in T^1\widetilde{M}$ and $x=p(y)$, we define $B^p_n(x,r)$ as the image under $p$ of the $(n,r)$-dynamical ball in $T^1\widetilde{M}$ centered at $y$. We say that $B^p_n(x,r)$ is the $p(n,\e)$-dynamical ball centered at $x$, where $p$ stands for projection. 
\end{definition}
 We recall that $N_\mu(n,\e,\delta)$ is the minimum number of $(n,\e)$-dynamical balls needed to cover a set of measure strictly bigger than $1-\delta$. We define $N_\mu^p(n,\e,\delta)$ as the  minimum number of $p(n,\e)$-dynamical balls needed to cover a set of measure strictly bigger than $1-\delta$. We also use the notation $N^p(C,n,\e)$ to denote the minimum number of $p(n,\e)$-dynamical balls needed to cover the set $C\subset T^1M$. To simplify notation we define $X=T^1M$. To prove the main result of this subsection we will need the following lemma, which follows easily from the criterion given in \cite{sp}.
\begin{lemma} Let $(M,g)$ be a pinched negatively curved manifold. Then there exists a H\"older continuous potential $\phi$ which admits an equilibrium state. 
\end{lemma}

\begin{theorem}[Simplified entropy inequality]\label{indeprad} Let $(M,g)$ be a pinched negatively curved manifold. Then for every ergodic measure $\mu$ have $$h_\mu(g)\le \liminf_{n\to\infty} \frac{1}{n}\log N^p_\mu(n,r,\delta),$$
where $\delta\in (0,1)$ and $r>0$. 
\end{theorem}
\begin{proof}  Fix $\delta\in (0,1)$, $r>0$ and $r'\in (0,r)$. Let $\phi$ be a H\"older continuous potential that admits an equilibrium state. Choose $m\in\N$ such that $1-\delta>\frac{1}{m}$. Let $F_n\subset X$  be a set satisfying $N^p_\mu(n,r,1-\frac{1}{m})=N^p(F_n,r,1-\frac{1}{m})$ and $\mu(F_n)>\frac{1}{m}$.  
By Birkhoff ergodic theorem there exists $F'\subset X$ and $N_0>0$ such that  $\mu(F')>1-\frac{1}{8m}$,  and $|\frac{1}{n}\int_0^{n-1}\phi(g_t x)dt-\int \phi d\mu|<\epsilon$, for every $x\in F'$ and $n\ge N_0$. From now on we will assume $n\ge N_0$. Let $K$ be a compact subset such that $\mu(K)>1-\frac{1}{8m}$. We will need the following fact, which follows directly from the formula $$\mu(A_1\cup...\cup A_n)=\sum_{i=1}^{n-1}(-1)^{i+1}\sum \mu(A_{t_1}\cap ...\cap A_{t_i}).$$
\begin{lemma} Let $F$ be a measurable set satisfying $\mu(F)>\frac{1}{s}$. Then for every $h\in\Z$ there exists $k\in [h,h+2s)$ such that $\mu(F\cap g_{-k}F)>\frac{1}{2^{2s}}$.
\end{lemma}
Define $S_n=F_n\cap K\cap F'$ and observe that by construction $\mu(S_n)>\frac{1}{2m}$. Then there exists $k_n\in (-4m,0]$ such that $$\mu(S_n\cap g_{-(n-1+k_n)}S_n)>\frac{1}{2^{4m}}.$$
Define $A_n=S_n\cap g_{-(n-1+k_n)}S_n$. In particular we get $$N^p_\mu(n,r',1-\frac{1}{2^{4m}})\le N^p(A_n, n, r').$$
Consider a maximal $p(n,2r)$-separated set of $A_n$ and denote by $R$ the set of centers of such dynamical balls. Observe that $\#R\le N^p(A_n,n,r)$.  
For each $x\in R$, let $E_x$ be a $p(n,r')$-separated set of maximal cardinality in $B^p_n(x,2r)$. By definition, the $p(n,r'/2)$-dynamical balls with centers in $E_x$ are disjoint. Moreover, since $\#E_x$ is maximal, the collection of $p(n,r')$-dynamical balls having centers in $E_x$ is a $p(n,r')$-covering of $B^p_n(x,2r)$. Define $Y=\bigcup_{l=0}^{4m} g_l  K$. Since $K$ is compact, the same holds for $Y$. Recall that $\phi$ is a H\"older potential that admits an equilibrium state,  denote its equilibrium measure by $m$. Therefore

\begin{equation*}
 \sum_{y\in E_x}m(B^p_n(y,r'/2))=m\left(\bigcup_{y\in E_x} B^p_n(y,r'/2) \right) \leq m(B^p_n(x,2r+r')),
\end{equation*}
and so
$$\#E_x \leq \frac{m(B^p_n(x,2r+r'))}{\min_{y\in E_x}m(B^p_n(y,r'/2))}.$$
Observe that by construction if $x\in A_n$, then $g_{n-1+k_n}(x)\in S_n$, which implies $g_n(x)\in Y$. In particular $A_n\subset Y\cap g_{-n}Y$. Recall that $m$ satisfies the Gibbs property, i.e. there exists a constant $C=C(Y,2r+r',r'/2)$ such that 
$$C^{-1}\le \frac{m(B^p_n(y,r_0))}{\exp(\int_0^{n}\phi(g_ty)dt-nP(\phi))}\le C,$$
for every $y\in Y\cap g_{-n}Y$, and $r_0\in\{2r+r',r'/2\}$ (for a precise definition of the Gibbs property see Definition \ref{gibbsprop}). Using the notation above this implies the bound
$$\#E_x \leq C^2 \exp\left(\int_0^{n}\phi(g_t x)dt- \min_{y\in E_x} \int_0^{n}\phi(g_ty)dt \right).$$
Therefore, by the definition of $F'$, we have
$$\#E_x \leq C' \exp(2n\epsilon).$$
Observe that
\begin{align*}
N_\mu^p(n,r',1-\frac{1}{2^{4m}}) &\leq N^p(A_n, n,r') \leq C'\exp(2n\epsilon) \#R \le C'\exp(2n\epsilon)N^p(A_n,n,r) \\
&\le C'\exp(2n\epsilon)N^p(F_n,n,r)= C'\exp(2n\epsilon)N^p_\mu(n,r,1-\frac{1}{m})\\
&\le C'\exp(2n\epsilon)N^p_\mu(n,r,\delta).
\end{align*}
We remark that $C'$ is independent of $n$; it only depends  on $r$, $r'$, $m$ and $K$. Then $$\liminf_{n\to\infty} \frac{1}{n}\log N^p_\mu(n,r',1-\frac{1}{2^{4m}})\le \liminf_{n\to\infty} \frac{1}{n}\log N^p_\mu(n,r,\delta)+2\epsilon.$$
Since $\epsilon>0$ was arbitrary we get 
$$\liminf_{n\to\infty} \frac{1}{n}\log N^p_\mu(n,r',1-\frac{1}{2^{4m}})\le \liminf_{n\to\infty} \frac{1}{n}\log N^p_\mu(n,r,\delta).$$
  By definition $B^p_n(x,r)\subset B_n(x,r)$. This implies that $N_\mu(n,s,q)\le N^p_\mu(n,s,q)$, for every $n\in \N$, $s\in \R$ and $q\in (0,1)$. Using Theorem \ref{ri} we get $$h_\mu(g)=\lim_{r'\to 0}\liminf_{n\to\infty} \frac{1}{n}\log N_\mu(n,r',1-\frac{1}{2^{4m}})\le \lim_{r'\to 0}\liminf_{n\to\infty} \frac{1}{n}\log N^p_\mu(n,r',1-\frac{1}{2^{4m}}).$$
Combining the inequalities above we obtain $$h_\mu(g)\le \liminf_{n\to\infty} \frac{1}{n}\log N^p_\mu(n,r,\delta).$$

\end{proof}

\section{Semicontinuity of the entropy map} \label{5}
In this section we will prove the upper semicontinuity of the entropy map. More precisely we will prove that if $(\mu_n)_n$ converges in the weak* topology to $\mu$, then $$\limsup_{n\to\infty}h_{\mu_n}(g)\le h_\mu(g).$$
In order to do so we will prove a more general inequality that involves the escape of mass and the topological entropy at infinity of the geodesic flow (see Section \ref{5.2}). We will also prove that the topological entropy at infinity  coincides with a measure theoretic counterpart--this is what we call the variational principle for the entropy at infinity. The results in Section \ref{5.2} and Section \ref{varprin} will have many consequences to the existence and stability of equilibrium states (see Section \ref{6}).  

\subsection{A more general result}\label{moregen} In this section we will prove a more general version than the one we need for the geodesic flow. This has the advantage of emphasizing the key ingredients to obtain  similar results. 

We start with our definition of topological entropy at infinity. Let $(X,d)$ be a non-compact topological manifold, and $T:X\to X$   an $L$-Lipschitz homeomorphism, i.e. a homeomorphism satisfying $d(Tx,Ty)\le Ld(x,y)$, for every $(x,y)\in X\times X$.  Let $K$ be a compact subset of $X$. Define $$K(n)=K\cap \bigcap_{i=1}^{n-2}T^{-i}K^c\cap T^{-(n-1)}K.$$
Given a point $x\in K$ and $r>0$, we define $C(x,n,r)$ as the number of $(n,r)$-dynamical balls needed to cover $B(x,r)\cap K(n)$.
We define the \emph{topological entropy outside $K$ at scale} $r$ by the formula $$\delta_\infty(K,r)=\limsup_{n\to\infty}\frac{1}{n}\log \sup_{x\in K}C(x,n,r),$$
and the \emph{topological entropy outside $K$} by $\delta_\infty(K)=\inf_{r>0} \delta_\infty(K,r)$.

\begin{definition}[Topological entropy at infinity]\label{teig} The \emph{topological entropy at infinity} of $(X,d,T)$ is the quantity $$\delta_\infty=\inf_{\{K_n\}}\liminf_{n\to \infty}  \delta_\infty(K_n ),$$
where the infimum runs over sequences $\{K_n\}_{n\in \N}$ where each $K_n$ is compact, $K_n\subset K_{n+1}$ and $X=\bigcup_{n\ge 1}K_n$.
\end{definition}

Our goal is to prove the following result.

\begin{theorem}\label{teoB} Let $(X,d)$ be a non-compact topological manifold and $T$ an $L$-Lipschitz homeomorphism with finite topological entropy at infinity. Assume that $(X,d,T)$ satisfies a simplified entropy inequality and that its ergodic measures are weak entropy dense.  Let $(\mu_n)_{n}$ be a sequence of $T$-invariant probability measures converging to $\mu$ in the vague topology.  Then 
$$\limsup_{n\to \infty} h_{\mu_n}(T)\le |\mu|h_{\mu/|\mu|}(T)+(1-|\mu|)\delta_\infty.$$
If the sequence $(\mu_n)_{n}$ converges to the zero measure the right hand side of the inequality is understood as $\delta_\infty$.
\end{theorem}
 
As mentioned in the introduction, we start with a weaker version of Theorem \ref{teoB}.

\begin{proposition} \label{pree} Let $(X,d)$ be a non-compact topological manifold and $T$ a $L$-Lipschitz homeomorphism with finite topological entropy at infinity. Assume that $(X,d,T)$ satisfies a simplified entropy inequality.  Let $(\mu_n)_{n}$ be a sequence of ergodic measures converging to $\mu$ in the vague topology. Suppose there exists a compact set $K$ such that  $\mu(\partial K)=0$, and $\mu_n(K)>0$, for every $n\in \N$.  Then 
$$\limsup_{n\to \infty} h_{\mu_n}(T)\le |\mu|h_{\mu/|\mu|}(T)+(1-\mu(A(K)))\delta_\infty(K,r),$$
where $A(K)=\{x\in X: T^k x\in K, \text{ for some } k\ge 0\text{ and some }k\le 0\}$.
\end{proposition}

\begin{proof}  Define $Y=A(K)$ as the set of points in $X$ that enter to $K$ under positive and negative iterates of $T$. Let $A_k$ be the set of points in $K$ that have their first return to $K$  at time $k$. Given $x\in Y$ define $n_2(x)$ as the smallest non-negative number such that $T^{n_2(x)}x\in K$ and $n_1(x)$ as the smallest non-negative number such that there exists $y\in K$ satisfying $T^{n_1(x)}(y)=x$. For $x\in Y$ define $n(x)=n_1(x)+n_2(x),$ and declare $n(x)=\infty$ whenever $x\in Y^c$. For $n\in\Z_{\ge 0}\cup\{\infty\}$ define $$C_n=\{x\in X: n(x)=n\}.$$ By the ergodicity of $\mu_m$ and the hypothesis $\mu_m(K)>0$ we have that $\mu_m(C_\infty)=0$. Moreover,  $x\in C_n$ means that $x$ is in the orbit of a point in $A_n$. Since $T$ is Lipschitz and $\sup_{x\in K} d(x,Tx)$ is finite, we conclude that $\bigcup_{n=0}^M C_n$ is bounded, and therefore relatively compact. Define $$\alpha_{N,M}=\bigcup^M_{n> N} C_n, \text{ and }\alpha_N=\bigcup_{n> N} C_n.$$ 
It worth mentioning that $\partial \alpha_N\subset \bigcup_{k\in \Z} T^{-k}\partial K$, and that the same holds for $\alpha_{N,M}$ and $C_\infty$. This implies that $\mu(\partial \alpha_{N,M})=\mu(\partial \alpha_N)=\mu(\partial C_\infty)=0$. By the definition of $\delta_\infty(K,r)$  we know that  given $\epsilon>0$, there exists $N_0=N_0(\epsilon)$ such that the following holds. For every $n\ge N_0$ and $x\in K(n)=K\cap \bigcap_{s=1}^{n-2}T^{-s}K^c  \cap T^{-(n-1)}K$, we have that $B(x,r)\cap K(n)$ can be covered with at most $e^{n(\delta_\infty(K,r)+\e)}$ $(n ,r)$-dynamical balls.

Choose natural numbers $k\ge 2$ and $N\ge N_0(\e)$. Define the partition  $$\beta_{k,N}=\{\alpha_{kN},\alpha_{N,kN}, Q^1_N,...,Q^s_N,C_\infty\},$$
where $Q_N^i$ are disjoint sets of diameter less than $r/(2L^{kN+2})$ covering $\bigcup_{n=0}^{N}C_n$ (which is relatively compact) and let $\beta_{k,N}'=\{Q^1_N,..., Q^s_N\}$. We choose this covering such that $\mu(\partial Q_N^i)=0$ for every $i$. In particular we know $\mu(\partial \beta_{k,N})=0$.  

Recall that for a partition $\P$ we denote $\P^n$ to the partition $\bigvee_{i=0}^{n-1}T^{-i}\P$. Let $Q\in \beta_{k,N}^n$ be such that $Q\subset \alpha_N^c$. We say that $[r,s)\subset [0,n)$ is an \emph{excursion} of $Q$ into $\alpha_{N}$ if $T^t Q \subset \alpha_{N}$ for every $t\in [r,s)$, $T^{r-1} Q\subset\alpha_N^c$ and $T^{s}Q\subset \alpha_N^c$.  Define $m_{k,N,n}(Q)$ as the number of  excursions of $Q$  into $\alpha_{N}$ that contain at least one excursion into $\alpha_{kN}$ and let $|E_{N,n}(Q)|=\#\{k\in[0,n): T^k Q\subset \alpha_N\}$. 

\begin{remark}\label{iteracion} \noindent
\begin{enumerate} 
\item\label{1s} Observe that if $x\in   C_n\cap K^c$, then $Tx\in C_n$, or $Tx\in K$. In other words if $x\in K^c$, then $n(T x)\le n (x)$.  

\item\label{2s} Let $[r,s)$ be an excursion of $Q$ into $\alpha_N$. Suppose there exists $x\in T^{r-1}Q\cap K^c$, then by (\ref{1s}) we have $n(x)\ge n(T x)>N$. In particular $x\in \alpha_N$, which contradicts that $T^{r-1}Q\subset \alpha_N^c$. We conclude that $T^{r-1}Q\subset K$. 

\item\label{3s} If $x\in Q$ and $Q\subset \alpha_N^c$, then $T^t x\in K$, for some $t\in [0,N]$. Indeed, if $x\in \bigcup_{n=0}^N C_n$, then by definition $n_2(x)\le N$ and the conclusion follows. 
\end{enumerate}
\end{remark}
 We claim that an atom $Q\in \beta_{k,N}^n$ such that $Q\subset K\cap T^{-(n-1)}K$, can be covered with an appropriate  number of $(n,r)$-dynamical balls. This estimate is very important in order to prove Proposition \ref{pree}. 

\begin{proposition} \label{22223} Let $\beta_{k,N}$ as above. Then an atom $Q\in \beta_{k,N}^n$ such that $Q\subset K\cap T^{-(n-1)}K$ can be covered with at most 
$$C_0e^{|E_{N,n}(Q)|(\delta_\infty(K,r)+\epsilon)}e^{m_{k,N,n}(Q)N(\delta_\infty(K,r)+\epsilon)},$$
$(n,r)$-dynamical balls, where $C_0=C_0(m,q,N,k)$.\end{proposition}

To simplify notation we will forget the subindex $N$ and $k$. We remark that since $C_\infty$ satisfies $T C_\infty\subset C_\infty$, the assumption $T^{n-1}Q\subset K $ rules out the possibility that $Q$ entered to $C_\infty$ before the $(n-1)$th iterate. 
The proof is inductive, and it is essentially the same idea as the one employed in \cite{ekp}. First decompose $[0,n-1]$ as $$[0,n-1]=W_1\cup V_1\cup... \cup V_s\cup W_{s+1},$$
according to the  excursions into $\alpha_{N}$ that contain at least one excursion into $\alpha_{kN}$. It worth mentioning that by Remark \ref{iteracion} each excursion into $\alpha_N$ can contain at most one excursion into $\alpha_{kN}$. More precisely, let $V_i=[n_i,n_i+h_i)$ and $W_i=[l_i,l_i+L_i)$ with $l_i+L_i=n_i$ and $n_i+h_i=l_{i+1}$.  The segment $V_i$ denotes an  excursion into $\alpha_{N}$ that contains  an excursion into $\alpha_{kN}$ and $(W_i)_i$ are the complementary intervals. \\\\
Step 0: Cover $\alpha_N^c=\bigcup_{n=0}^NC_n$ with  balls of diameter $r/L^{kN+2}$. By the definition of the set $\beta_{k,N}'$ we can take one ball per element in $\beta_{k,N}'$.  
We denote the number of balls required for this covering as $C_0=C_0(K, r,N,k)$. \\\\
Step 1: Assume we have covered $Q$ with $$C_0C_1^{i-1} e^{(\delta_\infty(K,r)+\epsilon)(|V_1|+...+|V_{i-1}|)}e^{(i-1)N(\delta_\infty(K,r)+\e)},$$ $(l_i+1,r)$-dynamical balls. We claim the same number of balls suffices to cover $Q$ with $(l_i+L_i,r)$-dynamical balls. Observe that by hypothesis $T^{l_i}Q\subset \alpha_N^c$, therefore $\text{diam }T^{l_i}Q\le r/L^{kN+2}$. Since the balls used to cover $\beta'$ have all diameter smaller than $r/L^{kN+2}$ the same hold if $Q$ spends some extra time in $\beta'$. If $Q$ have an excursion into $\alpha_N$ that does not enter to $\alpha_{kN}$, then by definition it must come back to $\beta'$ before $kN$ iterates (see Remark \ref{iteracion} (3)). In particular if the excursion into $\alpha_N$ is $[p_i,p_i+q_i)$, then $q_i\le kN$. Observe that  $\text{diam }T^{p_i -1}Q\le r/L^{kN+2}$ implies that $\text{diam }T^{p_i+t}Q\le r$ for every $t\in [0,kN)$. In particular the same holds for $t\in [0,q_i]$. Now we have entered to $\beta'$ again and we can repeat this process until we find an excursion into $\alpha_{kN}$, in that case we proceed to Step 2.
\\\\
Step 2:  Assume we have covered $Q$ by $$C_0e^{(\delta_\infty(K,r)+\epsilon)(|V_1|+...+|V_{i-1}|)}e^{(i-1)N(\delta_\infty(K,r)+\e)},$$ $(n_i,r)$-dynamical balls. To get a covering with $(n_i+h_i+1,r)$-dynamical balls we will cover each $(n_i,r)$-dynamical ball in the given covering by $(n_i+h_i+1,r)$-dynamical balls. Let $x\in Q$ be the center of one of the $(n_i,r)$-dynamical balls (if the center of the ball is not in $Q$ one takes a point in the ball that do belong to $Q$ and change $r$ by $2r$ in our next argument, for simplicity we assume that $x\in Q$ is the center of the dynamical ball). By definition $T^{t}x\in \alpha_{N}$ for $t\in [n_i,n_i+h_i)$,  $T^{n_i-1}x\in \alpha_N^c$ and  $T^{n_i+h_i}x\in \alpha_N^c$. Let $s(x)\ge 0$ be the smallest number such that $T^{n_i+h_i+s(x)}(x)\in K$. Notice that by Remark \ref{iteracion}($2$) we have $T^{n_i-1}(x)\in K$, and by Remark \ref{iteracion}($3$) we know that $s(x)\le N$. Since $T^{n_i-1} B_{n_i}(x,r)\subset B(T^{n_i-1}x,r)$, we can just focus on covering $B=B(T^{n_i-1}x,r)\cap K(h_i+s(x)+1)$ with $(h_i+1,r)$ dynamical balls. By the definition of $\delta_\infty (K,r)$ we know that $B $ can be covered with at most $e^{(\delta_\infty(K,r)+\e)(h_i+s(x))}$ $(h_i+s(x)+1,r)$-dynamical balls.  We conclude that $B$ can be covered by at most $e^{(\delta_\infty(K,r)+\e)(h_i+N)}$ $(h_i+1,r)$-dynamical balls. This proves that the number of $(n_i+h_i+1,r)$-dynamical balls needed to cover $Q$ is at most the number of balls we had at the beginning of Step 2, times $e^{(\delta_\infty(K,r)+\e)h_i}e^{(\delta_\infty(K,r)+\e)N}$.

We conclude that $Q$ can be covered with at most $$C_0e^{|E_{N,n}(Q)|(\delta_\infty(K,r)+\epsilon)}e^{m_{k,N,n}(Q)N(\delta_\infty(K,r)+\epsilon)},$$ $(n,r)$-dynamical balls, where $C_0=C_0(m,q,N,k)$.  We remark that $C_0$ is a constant independent  of $n$. We also remark that the term $|E_{N,n}(Q)|$ is a very rough bound, we can actually use the time spent in excursions into $\alpha_N$ containing excursions into $\alpha_{kN}$--for our applications this does not seem to be necessary.

\begin{proposition}\label{prop1111} Let $\beta_{k,N}$ the partition defined in Proposition \ref{22223}. Let $\mu$ be an ergodic T-invariant probability measure satisfying $\mu(K)>0$. Then
$$h_{\mu}(T)\le h_{\mu}(T,\beta_{k,N})+\mu(\alpha_N)(\delta_\infty(K,r)+\epsilon)+\frac{1}{k}(\delta_\infty(K,r)+\e).$$
\end{proposition}

\begin{proof}
Recall that by Definition \ref{sef} we know that for every ergodic measures $\mu$ we have 
$$h_\mu(T)\le \lim_{n\to \infty}\frac{1}{n}\log N_\mu(n,r,\delta).$$
Using the ergodicity of $\mu$ and the assumption $\mu(K)>0$ we can find an increasing  sequence $(n_i)_{i}$ such that $$\mu(K \cap T^{-n_i}K) > \delta_1,$$
for every $i\in \N$, where $\delta_1$ is  sufficiently small but positive (and independent  of $n_i$). By Shannon-McMillan-Breiman theorem the set
$$A_{\e_1,N}=\{x\in X : \forall n\geq N, \mu(\beta^n(x))\geq \exp(-n(h_\mu(T,\beta)+\e_1))\},$$
has measure converging to 1 as $N$ goes to $\infty$, for every $\e_1>0$. Fix $\epsilon_1>0$ small. By Birkhoff ergodic theorem there exists a set $W_{\e_1}$ such that
$$\frac{1}{n}\sum_{i=0}^{n} 1_{\alpha_N}(T^n x) < \mu(\alpha_N)+\e_1,\text{ and }\mu(W_{\e_1})>1-\frac{\delta_1}{4},$$
for all $x\in W_{\e_1}$ and $n\ge n(\e_1)$. We finally define
$$X_i= W_{\e_1}\cap  A_{\e_1,n_i}\cap K\cap T^{-n_i}K.$$
By construction, for $i$ sufficiently large, we have $\mu(X_i)>\frac{\delta_1}{2}$. From now on we will always assume $i$ is sufficiently large. Our goal is to cover $X_i$ by $(n_i,r)$-dynamical balls. By definition of $A_{\e_1,n_i}$ we know $X_i$ can be covered with $\exp(n_i(h_\mu(T,\beta)+\e_1))$ many elements of $\beta^{n_i}$. We will use Proposition \ref{22223} to cover efficiently each of those atoms with dynamical balls. Let $Q\in \beta^{n_i}$ be an atom intersecting $X_i$, in particular $Q\in K\cap T^{-(n-1)}K$. By the choice of $W_{\e_1}$  we have $$|E_{N,n_i}(Q)|<(\mu(\alpha_N)+\e_1)n_i.$$ 
We claim that $m_{k,N,n_i}(Q)\le \frac{1}{Nk}n_i$. This follows from Remark \ref{iteracion}. Let  $[p,p+q)$ be is an excursion of $Q$ into $\alpha_N$ that contain an excursion into $\alpha_{kN}$. There exists a smallest $h\ge 0$ such that $T^{p+q+h}Q\subset K$. By definition of $\alpha_{kN}$ we have $q+h+1\ge kN$. Moreover $T^kQ \subset \alpha_N^c$ for every $k\in [p+q+1,p+q+h]$. In particular each excursion into $\alpha_{kN}$ generates an interval of length at least $kN$ where no other excursion into $\alpha_{N}$ can occur.

  Putting all together, we get that $N(n_i,r,1-\frac{\delta_1}{2})$ is bounded  above by
$$e^{n_i(h_\mu(T,\beta)+\e_1)}C_0  e^{n_i(\delta_\infty(K,r)+\epsilon)(\mu(\alpha_N)+\e_1)}e^{\frac{1}{kN}n_iN(\delta_\infty(K,r)+\e)}.$$
Finally we obtained
\begin{align*}
h_\mu(T)\le & h_\mu(T,\beta_{k,N})+\e_1+(\delta_\infty(K,r)+\e)(\mu(\alpha_N)+\e_1)+\frac{1}{k}(\delta_\infty(K,r)+\e).
\end{align*}
Since $\epsilon_1>0$ was arbitrary we are done.  
\end{proof}

We now explain how to get Proposition \ref{pree} from Proposition \ref{prop1111}; this is basically identical to the last part of the proof of the main result in \cite{ekp}. First assume $\mu(X)>0$, and fix  $\varepsilon_0>0$. We remark that by construction $\mu(\partial \beta_{k,N})=0$.  To simplify notation we use $\beta$ instead of $\beta_{k,N}$ to denote our partition. Choose $m$ sufficiently large such that 
$$h_{\frac{\mu}{|\mu|}}(T)+\varepsilon_0>\frac{1}{m}H_{\frac{\mu}{|\mu|}}(\beta^{m}), \quad 2\frac{e^{-1}}{m}<\frac{\varepsilon_0}{2},$$
and $-(1/m)\log \mu(X)<\varepsilon_0$. Then $$|\mu| h_{\frac{\mu}{|\mu|}}(T)+2\varepsilon_0>\frac{1}{m}\sum_{P\in \beta^m} \mu(P)\log\mu(P).$$
Define $A=\bigcap_{i=0}^{m-1}T^{-i} \alpha_{kN}$ and observe that by the definition of the vague convergence we have
$$\lim_{n\to\infty} \sum_{Q\in \beta^{m}\setminus\{A\}}\mu_n(Q)\log\mu_n(Q) = \sum_{Q\in \beta^{m}\setminus\{A\}}\mu(Q)\log\mu(Q).$$
Choosing $n$ sufficiently large we get the inequality 
$$|\mu|h_{\frac{\mu}{|\mu|}}(T)+3\varepsilon_0\ge \frac{1}{m}H_{\mu_n}(\beta^m).$$
Finally use Proposition \ref{prop1111} to get 
\begin{align*}
\mu(X)h_{\frac{\mu}{|\mu|}}(T)+3\varepsilon > & \frac{1}{m}H_{\mu_n}(\beta^{m})\geq h_{\mu_n}(T,\beta)\\
\ge & h_{\mu_n}(T)-(\delta_\infty(K,r)+\e)\mu_n(\alpha_N) -\frac{1}{k}(\delta_\infty(K,r)+\e).
\end{align*}
Observe that $\bigcup_{n=0}^NC_n$ is relatively compact and that $\mu_n(\alpha_N)=1-\mu_n(\bigcup_{n=0}^NC_n)$. We remark that by construction $\mu(\partial \alpha_N)=0$. Therefore
\begin{align*} \limsup_{n\to\infty} h_{\mu_n}(T)\leq & \mu(X)h_{\frac{\mu}{|\mu|}}(T)+(\delta_\infty(K,r)+\e)(1-\mu(\bigcup_{n=0}^NC_n))\\ & +\frac{1}{k}(\delta_\infty(K,r)+\e).
\end{align*}
Observe that by construction we  can send $\e$ to zero as $N$ goes to infinity. Finally take $k\to \infty$ and $N\to \infty$. We obtained the desired inequality
$$\limsup_{n\to \infty} h_{\mu_n}(T)\le |\mu|h_{\frac{\mu}{|\mu|}}(T)+(1-\mu (A(K)))\delta_\infty(K,r).$$

The case when $\mu(X)=0$ follows directly from Proposition \ref{prop1111} since $h_{\mu_n}(g,\beta)\to 0$ and $\mu_n(\alpha_N)=1-\mu_n(\bigcup_{s=1}^NC_s)\to 1$ as $n$ tends to $\infty$.
\end{proof}


We have finally all the ingredients to prove Theorem \ref{teoB}.
\begin{proof}[Proof of Theorem \ref{teoB}]
We will prove that 
\begin{align}\label{1op} \limsup_{n\to \infty} h_{\mu_n}(T)\le |\mu|h_{\mu/|\mu|}(T)+(1-\mu(A(K)))\delta_\infty(K,r),\end{align}
holds, for every sufficiently large compact set $K$ and $r>0$. Let $\mu_0$ be a $T$-invariant measure with finite entropy that gives positive measure to $K$ (which exists because $K$ is sufficiently large) and define $\mu'_{n}=(1-\frac{1}{n})\mu_n+\frac{1}{n}\mu_0$. By hypothesis, the space of ergodic measures is weak entropy dense (see Definition \ref{wed}), therefore we can find an ergodic measure $\nu_n$ arbitrarily close to $\mu_{n}'$ (in the weak* topology) such that $h_{\nu_n}(T)>h_{\mu'_n}(T)-\frac{1}{n}$. In particular we can assume $\nu_n(K)>0$ and that the sequence $(\nu_n)_n$ converges to $\mu$ in the weak* topology. We can now use Proposition \ref{pree} to the sequence $(\nu_n)_{n}$ and get
$$\limsup_{n\to \infty} h_{\nu_n}(T)\le |\mu|h_{\mu/|\mu|}(T)+(1-\mu(A(K)))\delta_\infty(K,r).$$
By construction we have  $$h_{\nu_n}(T)>h_{\mu'_n}(T)-\frac{1}{n}=\bigg(1-\frac{1}{n}\bigg)h_{\mu_n}(T)+\frac{1}{n}h_{\mu_0}(T)-\frac{1}{n},$$
and therefore $$\limsup_{n\to \infty} h_{\nu_n}(T)\ge \limsup_{n\to \infty} h_{\mu_n}(T),$$
which implies the inequality (\ref{1op}). Take an increasing sequence $(K_i)_{i}$ of compact sets such that $\lim_{i\to\infty}\delta_\infty(K_{i})=\delta_\infty$. Now observe that $A(K_{i})\subset A(K_{i+1})$ and $\bigcup_{i\ge 1}A(K_i)=X.$ This implies that  $\lim_{i\to \infty} \mu(A(K_i))=\mu(X)$, and therefore inequality (\ref{1op}) finishes the proof. 
\end{proof}

\subsection{Upper semicontinuity of the entropy map for the geodesic flow}\label{5.2}
For the geodesic flow the topological entropy at infinity introduced in Definition \ref{teig} takes a more concrete form. We will explain the  modifications in the notation and proof of Theorem \ref{teoB} when dealing with the geodesic flow. For consistency with the notation used in Section \ref{moregen} we set $X=T^1M$, and $T=g_1$, the time one map of the geodesic flow.  For a set $Z\subset X$ we define $N_1(Z)=\{x\in X: d(x,Z)\le 1\}$. 

For a compact set $K\subset T^1M$ and $t\ge 3$ we define $$K(t)=\{x\in N_1(K): g_s(x)\in K^c, \text{ for }s\in[1,t-2]\text{ and }g_{t-1}(x)\in K\}.$$
Given a point $x\in N_1(K)$, and $r>0$ we define $C(x,t,r)$ as the number of $p(t,r)$ dynamical balls needed to cover $B(x,r/4)\cap K(t)$ (we slightly changed the definition of $C(x,t,r)$ by shrinking the ball of radius $r$ to $r/4$; this does not affect any part in the proof of Theorem \ref{teoB}). Then define $$\delta_\infty(K)=\inf_{r>0}\limsup_{t\to\infty}\frac{1}{t}\log \sup_{x\in N_1(K)}C(x,t,r).$$  
We remark that in the definition of $K(t)$ we consider points in $N_1(K)$--rather than  points in $K$--because of Remark \ref{iteracion2}(\ref{2ss}). 
 Recall that in the proof of Proposition \ref{pree} we defined the function $n(x)=n_1(x)+n_2(x)$. We now allow $n_1(x)$ and $n_2(x)$ to be non-negative real numbers (instead of non-negative integers). In particular we have that
$$C_t=\{x\in X:n(x)=t\},$$
is well defined for every $t\in \R_{\ge 0}$. As before, $C_\infty$ is the complement of $A(K)$ in $X$. Similarly we modify the definition of $\alpha_N$ and $\alpha_{N,M}$ into 
$$\alpha_{N,M}=\bigcup_{s\in (N,M]}C_s,\text{ and }\alpha_{N}=\bigcup_{s\in (N,\infty)}C_s.$$
It still holds that $\alpha_N^c=\bigcup_{s\in [0,N]}C_s$ is relatively compact and that $\mu(\partial \alpha_N)=\mu(\partial \alpha_{N,M})=0$. In the proof of Proposition \ref{22223}, Remark \ref{iteracion} plays an important role (specially parts (\ref{2s}) and (\ref{3s})). For the geodesic flow Remark \ref{iteracion} becomes 
\begin{remark}\label{iteracion2} \noindent
\begin{enumerate} 
\item \label{1ss} We claim that  $\alpha^c_N\cap T^{-1}\alpha_N \subset N_1(K)$. Indeed, if $x\in \alpha^c_N$ and $T(x)=g_1(x)\in \alpha_N$, then there exists $s\in [0,1]$ such that $g_s(x)\in K$. This automatically implies that $x\in N_1(K)$
\item\label{2ss} Let $[r,s)$ be an excursion of $Q$ into $\alpha_N$. Then by (\ref{1ss}) we have that $T^{r-1}Q\subset N_1(K)$.

\item\label{3ss} If $x\in Q$ and $Q\subset \alpha_N^c$, then $g_t (x)\in K$, for some $t\in [0,N]$. Indeed, if $x\in \bigcup_{s\in [0,N]} C_s$, then by definition $n_2(x)\le N$ and the conclusion follows. 
\end{enumerate}
\end{remark}

We emphasize that part (\ref{3ss}) (resp. part (\ref{2ss})) is used in step 1 (resp. step 2) in the proof of Proposition \ref{22223}. With these modifications the proof of Theorem \ref{teoB} go through for the geodesic flow. An important point to make is that we are using $T=g_1$, the time one map of the geodesic flow. In particular to use Proposition \ref{pree} we need a sequence of ergodic measures under the action of $g_1$. By \cite[Lemma 7]{op} we can choose $t<1$ arbitrarily close to $1$ such that our sequence of ergodic flow invariant measures are ergodic under the action of $g_t$. This is enough to obtain Proposition \ref{pree} for any sequence of ergodic flow invariant probability measures satisfying $\mu_n(K)>0$.


For the geodesic flow it makes sense to consider a definition of the topological entropy at infinity that is closer to the definition of the critical exponent of $\pi_1(M)$. Before making this explicit let us recall some notation. 

We use the notation $d$ for the metric on $T^1M$ and $T^1\widetilde{M}$, and the notation $g_t$ for the geodesic flow on $T^1M$ and $T^1\widetilde{M}$ (the difference will be clear from the inputs of $d(\cdot, \cdot)$ and $g_t(\cdot)$). Also recall that $x,y\in T^1M$ are $p(n,r)$-separated if $x$ does not belong to the $p(n,r)$-dynamical ball centered at $y$. We say that $\tilde{x},\tilde{y}\in T^1 \widetilde{M}$ are $p(n,r)$-separated if $d_n(\tilde{x},\tilde{y})>r$. We have the canonical projection $p:T^1\widetilde{M}\to T^1M$ and $\pi$ is the projection that sends a vector to its base point (in $T^1M$ or $T^1\widetilde{M}$). We denote by $\widetilde{\Omega}$ to $p^{-1}(\Omega)$.

Let $Q\subset T^1\widetilde{M}$ be a compact subset and $W\subset T^1\widetilde{M}$ an open relatively compact  such that $W\cap \widetilde{\Omega}\ne \emptyset$. Let $D\in \N$. Given $n\in\N\cap [2D,\infty)$ we define
\begin{align*}\Gamma^W_Q(n,D)=\{\gamma\in\Gamma:\exists x\in W\cap \widetilde{\Omega}&\text{ such that }g_s(x)\in (\Gamma Q)^c,\text{ for }s\in [D,n-D] \\ &\text{ and }g_{u}(x)\in  \gamma W,\text{ for some }u\in [n-1,n]\}.\end{align*}
Then define \begin{align}\label{truedef}\delta_\infty^W(Q)=\sup_{D\in \N}\limsup_{n\to\infty}\frac{1}{n}\log \# \Gamma^W_Q(n,D).\end{align}




\begin{lemma}\label{lem:com}
Set $W_0=p(W)$, $K_0=p(Q)$, and $\widehat{Q}\subset Q$ be another compact set such that $d(\widehat{Q},\widetilde{M}\setminus Q)>0$. Then $\delta_\infty(K_0)\le \delta^W_\infty(\widehat{Q})$. 
\end{lemma}

\begin{proof} Pick a point $w\in W_0\cap \Omega$. Let $x\in N_1(K_0)$ and $r>0$ sufficiently small compared with the injectivity radius on $N_1(K_0)$. Let  $\widetilde{x}\in T^1\widetilde{M}$ be a lift of $x$ to $T^1\widetilde{M}$. Similarly, for every $y\in B(x,r)$ we denote by $\widetilde{y}\in T^1\widetilde{M}$ the lift of $y$ satisfying $d(\widetilde{x},\widetilde{y})\le r$. 
Let $S_t$ be a maximal  $p(t,r)$-separated set in $B(x,r/4)\cap K_0(t)$. Notice that $C(x,t,r)\le \#S_t$. For each element $s\in S_t$ we consider $q(s):=g_{t-1}(s)\in K_0$. 


Fix $\delta>0$. By the transitivity of the geodesic flow there exists $L=L(W_0,K_0,\delta)>0$ such that the following holds: for every $y\in K_0$ there exists $y_\delta$ such that $d(y,y_\delta)<\delta$ and such that $d(w,g_{j(y)} y_\delta)<\delta$, for some $j(y)\in [0, L]$.  Moreover, for $z\in N_1(K_0)$ there exists a point $w^z_\delta\in W_0$ such that $d(w,w^z_\delta)<\delta$ and $d(z,g_{j(w^z)}w^z_\delta)<\delta$, for some $j(w^z)\in [0,L]$. If $z=x$ we denote the point simply by $w_\delta$ and $j(w):=j(w^x)$. 


For every $s\in S_t$ we have three geodesic segments: the first is simply the geodesic $\{g_p(w_\delta)\}_{p\in [0,j(w)]}$, the second is $\{g_p(s)\}_{p\in [0,t-1]}$, and the third is $\{g_p(q(s)_\delta)\}_{p\in [0,j(q(s))]}$. By the local product structure and the closing lemma we can find a closed geodesic, say $c(s)$,  shadowing these three geodesic segments. By taking $\delta>0$ sufficiently small, and $t$ sufficiently large we can make sure the geodesic is $\epsilon$-close to the union of these three geodesic segment (for any fixed $\e>0)$--this argument was already used in the proof of Proposition \ref{entropydense}. 

The length of $c(s)$, say $l(s)$, is approximately the length of the three segments, in particular it belongs to $[t-2,t+2L]$. In fact, this whole approximation by a closed geodesics takes place in the universal cover, that is, if one lift the segments to the universal cover (so that the end of the first is sufficiently close to the beginning of the second, and the end of the second is sufficiently close to the beginning of the third), then there exists a hyperbolic axis $\e$-close to this chain of paths. For convenience we take the lift that passes near to $\widetilde{s}$. Choose a point $\widetilde{h(s)}$ tangent to the axis and such that $d(\widetilde{s},\widetilde{h(s)})<\e$. Denote by $\gamma_s\in \Gamma$ the hyperbolic element with such axis that represents $c(s)$. 

Let us suppose that $\gamma_s=\gamma_{s'}$ for distinct points $s,s'\in S_t$. If $t$ is sufficiently large we can ensure that  $d(\gamma_{s}$$\widetilde{s},g_{l(s)}\widetilde{h(s)})<2\e$; we will assume this is the case. Therefore 
$$d(\gamma_s \widetilde{x}, g_{l(s)}\widetilde{h(s)})\le  d(\gamma_s\widetilde{x},\gamma_s\widetilde{s})+d(\gamma_s\widetilde{s},g_{l(s)}\widetilde{h(s)}) \le r/4+\e.$$
Analogously we have $d(\gamma_{s'}\widetilde{x},g_{l(s')}\widetilde{h(s')})\le r/4+\e$.  Combining these two inequalities
\begin{align}\label{up1}d(g_{l(s)}\widetilde{h(s)},g_{l(s')}\widetilde{h(s')})\le r/2+2\e.\end{align}
 Also notice that 
\begin{align}\label{up2}d(\widetilde{h(s)},\widetilde{h(s')})\le d(\widetilde{h(s)},\widetilde{s})+d(\widetilde{s},\widetilde{x})+d(\widetilde{x},\widetilde{s'})+d(\widetilde{s'},\widetilde{h(s')})\le r/2+2\e.\end{align}
Since $s$ and $s'$ are $p(t,r)$-separated we know that $d_t(\widetilde{s},\widetilde{s'})>r$. It follows from the 
construction that $d_t(\widetilde{h(s)},\widetilde{s})<\e$. In particular 
\begin{align}\label{up3}d_t(\widetilde{h(s)},\widetilde{h(s')})\ge d_t(\widetilde{s},\widetilde{s'})-d_t(\widetilde{h(s)},\widetilde{s})-d_t(\widetilde{h(s')},\widetilde{s'})>r-2\e.\end{align}
It follows from inequalities (\ref{up1}) and (\ref{up2}), and the convexity of the distance function in $CAT(0)$ spaces that $d_t(\widetilde{h(s)},\widetilde{h(s')})\le r/2+2\e.$ If $\e<r/9$ this would contradict inequality (\ref{up3}).

From now on we assume $t\gg0$ and $\delta\ll 1$  are choosen to ensure that $\e<\min\{r/9,d(\widehat{Q},\widetilde{M}\setminus Q)\}$. Since $L$ depends only on $\delta$, $K_0$ and $W_0$, we consider it fixed from now on. As explained above, in this case the map $s\mapsto \gamma_s$ is an injection. The assumption $\epsilon<d(\widehat{Q},\widetilde{M}\setminus Q)$ ensures that  $\gamma_s\in \Gamma_{\widehat{Q}}(\lfloor l(s)\rfloor ,L+1+D)$. Indeed, choose a point $w_s\in W_0$ which is tangent to $c(s)$ and such that $d(w_s,w)<\epsilon$. A lift of $w_s$ to $W$ can be used in the definition of $\Gamma_{\widehat{Q}}(\lfloor l(s)\rfloor ,L+1+D)$ to check that $\gamma_s\in \Gamma_{\widehat{Q}}(\lfloor l(s)\rfloor,L+1+D)$. This follows directly from the construction of $c(s)$, and the definition of $K_0(t)$.


Recall that $l(s)\in [t-2,t+2L]$. We conclude that $$C(x,t,r)\le \#S_t\le \sum_{u=t-2}^{t+2L} \Gamma_{\widehat{Q}}(u,L+1+D).$$
Observe that the right hand side is independent of $x\in N_1(K_0)$ ($L$ is independent of the point in $N_1(K_0)$). This and the definition of $\delta_\infty(K_0)$ and $\delta_\infty^W(\widehat{Q})$ imply the inequality $\delta_\infty(K_0)\le \delta_\infty^W(\widehat{Q})$.
\end{proof}





Let $P$ be a Dirichlet fundamental domain of $M$ in $\widetilde{M}$.  

\begin{definition}[Topological entropy at infinity]\label{teidef} The \emph{topological entropy at infinity of the geodesic flow} is the quantity  $$\delta_\infty=\inf_{(P_k)}\liminf_{k\to \infty}\delta_\infty^W(P_k),$$
where the infimum runs over compact exhaustions of $P$, in other words, increasing sequences $(P_k)_{k}$ such that each $P_k$ is compact and $\bigcup_{n\ge 1}P_k=P$.
\end{definition}

It follows from the definition that  $\delta_\infty=\delta^0_{\Gamma,\infty}$, where $0$ is the zero potential (see formula (\ref{xxxxy})). That is, the topological entropy at infinity is simply the critical exponent at infinity of the zero potential. It  follows from the discussion right after Definition \ref{sprdef} that $\delta_\infty$ is independent of $W$. 


\begin{remark} Shortly after a first draft of this paper was finished, B. Schapira told us that in a joint work with S. Tapie \cite{st} they have also defined the topological entropy at infinity. The critical gap $\delta_\infty<h_{top}(g)$, has important consequence in the study of the regularity of the topological entropy under $C^1$ perturbations of the metric (see \cite{st}).
\end{remark}


It follows from Lemma \ref{lem:com} that the topological entropy at infinity defined in Theorem \ref{teoB} is at most the topological entropy at infinity of the geodesic flow. As an application of Theorem \ref{teoB} and the discussion above we obtain one of the main results of this paper

\begin{theorem}\label{A} Let $(M,g)$ be a pinched negatively curved manifold. Let $(\mu_n)_{n}$ be a sequence of invariant probability measures converging to $\mu$ in the vague topology. Then 
$$\limsup_{n\to \infty}h_{\mu_n}(g)\le |\mu|h_{\mu/|\mu|}(g)+(1-|\mu|)\delta_\infty.$$
If the sequence converges vaguely to zero, then the right hand side is understood as $\delta_\infty$. 
\end{theorem}

In Theorem \ref{varpri} we will prove that Theorem \ref{A} is sharp, in particular the topological entropy at infinity of the geodesic flow coincides with the one defined at the beginning of this section (which comes from Section \ref{moregen}). 

 It worth mentioning that if $M$ is geometrically finite, then $$\delta_\infty=\max_\P \delta_\P,$$ where the maximum runs over the parabolic subgroups of $\pi_1(M)$. This follows from the structure of the ends of $\Omega$ when $M$ is geometrically finite, and the convexity of the horoballs. In particular we recover one of the main results of \cite{rv}. 


The following definition was coined in \cite{st}, and fits perfectly within the analogy between the geodesic flow and countable Markov shifts. 

\begin{definition}[Strongly positive recurrent manifolds]\label{spr_def} We say a pinched negatively curved manifold $M$ is \emph{strongly positive recurrent} (SPR for short) if $$\delta_\infty<h_{top}(g).$$
\end{definition}

Among the consequences of Theorem \ref{A} we get that SPR manifolds have a measure of maximal entropy (see Theorem \ref{cme}). This definition should also be compared with Definition \ref{sprdef}, where an analogous critical gap condition is required. 

We remark that every hyperbolic geometrically finite manifold is SPR. This follows from the fact that in $\mathbb{H}^n$ every parabolic subgroup is of divergence type and \cite[Proposition 2]{dop}. 

A direct application of Theorem \ref{A} is the upper semicontinuity of $\mu\mapsto h_\mu(g)$, and $\mu \mapsto h_\mu(g)+\int Fd\mu$ (with respect to the weak* topology).

\begin{theorem}[Upper semicontinuity of the entropy map]\label{us}  Let $(M,g)$ be a pinched negatively curved manifold. Let $(\mu_n)_{n}$ be a sequence of invariant probability measures converging to $\mu$ in the weak* topology. Then 
$$\limsup_{n\to \infty}h_{\mu_n}(g)\le h_\mu(g).$$
\end{theorem}
\begin{theorem}[Upper semicontinuity of the pressure]\label{usp}  Let $(M,g)$ be a pinched negatively curved manifold. Let $(\mu_n)_{n}$ be a sequence of invariant probability measures converging to $\mu$ in the weak* topology and $F\in C_b(T^1M)$. Then 
$$\limsup_{n\to \infty}\big(h_{\mu_n}(g)+\int Fd\mu_n\big)\le \big(h_\mu(g)+\int F d\mu\big).$$
\end{theorem}

\begin{remark}After the work of Newhouse \cite{n} we know that on a compact manifold the entropy map is upper semicontinuous if the dynamics is of class $C^\infty$ (see also \cite{y}). Later on, Buzzi \cite{bu} refined Newhouse's result and proved that those dynamical systems are asymptotically h-expansive, which it is known to imply the upper semicontinuity of the entropy map in the compact setting. We remark that their methods use, in an essential way, the compactness of the manifold. 
 
\end{remark}

\subsection{Variational principle for the entropy and pressure at infinity}\label{varprin} Theorem \ref{A} indicates the importance of  the topological entropy at infinity when studying the regularity of the entropy map. We now recall its measure theoretic counterpart and prove that these two a priori different notions of entropy at infinity coincide.

\begin{definition}[Measure theoretic entropy at infinity]\label{meaent} The \emph{measure theoretic entropy at infinity} of the geodesic flow is defined as 
$$h_\infty=\sup_{(\mu_n)}\limsup_{\mu_n\to 0} h_{\mu_n}(g),$$
where the supremum runs over sequences of invariant probability measures converging vaguely to the zero measure.
\end{definition}
This invariant was first defined in the context of the geodesic flow in \cite{irv} and proved to be equal to the topological entropy at infinity for the geodesic flow on extended Schottky manifolds via symbolic methods. This result was later extended to cover all geometrically finite manifolds in \cite{rv}. In this paper we generalize those results to any pinched negatively curved manifold. We emphasize that $h_\infty=P_\infty(0)$ (see Definition \ref{defpresinf}). 

\begin{theorem}[Variational principle for the entropy at infinity]\label{varpri} The topological entropy at infinity is equal to the measure theoretic entropy at infinity. Equivalently, $\delta_\infty=h_\infty.$
\end{theorem}

For sake of generality we will prove some results for arbitrary potentials and eventually restrict to the zero potential to prove Theorem \ref{varpri}. We start with the following  useful lemma. 


\begin{lemma}\label{entinf11} Let  $F\in C_b(T^1M)$ be uniformly continuous and $\phi\in C_0(T^1M)$. Then $$P(F+\phi)\ge \delta^F_{\Gamma,\infty}.$$
\end{lemma}
\begin{proof} It is enough to prove the result for $F$ and $\phi$ H\"older continuous (by $C^0$ approximation we conclude the result for uniformly continuous potentials).  

Let $N$ be a Dirichlet fundamental domain of $M$ in $\widetilde{M}$. Given $\e>0$, there exists a compact subset $K=K(\e)$ of $M$ such that $\phi(x)\in [-\e,\e]$, for every $x\in T^1K^c$. Let $Q=Q(\e)$ be the subset of $N$ corresponding to $K$. Let $W$ be an open relatively compact set intersecting $\widetilde{\Omega}$ and $E$ its diameter.  Set $M=||\phi||_0$. 

Choose a reference point $z\in W$. Let $\gamma\in \Gamma^W_Q(n,D)$. In this case there exists a point $x^n_\gamma \in W$ such that $g_k(x^n_\gamma)\in (T^1(\Gamma Q))^c$, for $k\in [D,n-D]$ and $g_u(x^n_\gamma)\in \gamma W$, for some $u\in (n-1,n]$. Observe that $d(z,\pi(x^n_\gamma))$ and $d(\gamma z, \pi(g_u(x^n_\gamma)))$ are both bounded above by $E$. In particular, by Lemma \ref{lem:imp} there exist constants $N_1,N_2>0$ independents of $\gamma$ (but depending on $s$, since those depend on the potential we are using) such that 
$$|\int_z^{\gamma z}(\widetilde{F}+\widetilde{\phi}-s)-\int_{\pi(x^n_\gamma)}^{\pi(g_u(x^n_\gamma))}(\widetilde{F}+\widetilde{\phi}-s)|\le N_1,$$
and that 
$$|\int_z^{\gamma z}(\widetilde{F}-s)-\int_{\pi(x^n_\gamma)}^{\pi(g_u(x^n_\gamma))}(\widetilde{F}-s)|\le N_2.$$
By the definition of $Q$ and its connection with $\phi$ it follows that 
$$\int_{\pi(x^n_\gamma)}^{\pi(g_u(x^n_\gamma))}(\widetilde{F}+\widetilde{\phi}-s)\ge-M(2D+1)-\e(n-2D)+ \int_{\pi(x^n_\gamma)}^{\pi(g_u(x^n_\gamma))}(\widetilde{F}-s).$$
Set $M_0:=-M(2D+1)+\e2D)$. From the triangle inequality we have that $d(z,\gamma z)\in [n-2E-1, n+2E]$. In particular if $\tau\in \bigcap_{i=1}^k \Gamma^W_Q(n_i,D)$, then $d(z,\tau z)\in \bigcap_{i=1}^k[n_i-2E-1, n_i+2E]$. We conclude that $\tau\in \Gamma$ can belong to at most $4(E+1)$ different sets from $\{\Gamma^W_Q(m,D)\}_{m}$. Combining all this we get
 \begin{align*}4(E+1)\sum_{\gamma\in \Gamma} \exp(\int_z^{\gamma z} \widetilde{F}+\widetilde{\phi}-s)&\ge\sum_{n\ge2D}\sum_{\gamma\in \Gamma^W_Q(n,D)}\exp(\int_z^{\gamma z}\widetilde{F}+ \widetilde{\phi} -s)\\ 
&\ge e^{-N_1}\sum_{n\ge 2D}\sum_{\gamma\in \Gamma^W_Q(n,D)}\exp\big(\int_{\pi(x^n_\gamma)}^{\pi(g_u(x^n_\gamma))}(\widetilde{F}+\widetilde{\phi}-s)\big)\\
&\ge e^{-N_1+M_0}\sum_{n\ge 2D}e^{-\e n}\sum_{\gamma\in \Gamma^W_Q(n,D)}\exp\big(\int_{\pi(x^n_\gamma)}^{\pi(g_u(x^n_\gamma))}(\widetilde{F}-s)\big)\\
&\ge e^{-N_1+M_0-N_2}\sum_{n\ge 2D}e^{-\e n}\sum_{\gamma\in \Gamma^W_Q(n,D)}\exp\big(\int_{z}^{\gamma z}(\widetilde{F}-s)\big).\\
&\ge e^{-N_1+M_0-N_2+2\e-t(s)(2E+1)s)}\sum_{n\ge 2D}e^{-n(s+\e)}\sum_{\gamma\in \Gamma^W_Q(n,D)}\exp\big(\int_{z}^{\gamma z}\widetilde{F}\big),\\
 \end{align*}
where $t(s)$ is the sign of $s$. Summarizing we obtained that 
\begin{align}\label{tp}\sum_{\gamma\in \Gamma} \exp(\int_z^{\gamma z} \widetilde{F}+\widetilde{\phi}-s)\ge R_s \sum_{n\ge 2D}e^{-n(s+\e)}\sum_{\gamma\in \Gamma^W_Q(n,D)}\exp\big(\int_{z}^{\gamma z}\widetilde{F}\big),\end{align}
for some constant $R_s>0$ that only depends on $s$. 

Recall that by formula (\ref{xxxxy}) we know that 
\begin{align}\label{xxxx} \delta^F_\Gamma(W,Q,D)=\limsup_{n\to\infty}\frac{1}{n}\log \sum_{\gamma\in \Gamma^W_Q(n,D)}\exp\big(\int_z^{\gamma z}\widetilde{F}\big).\end{align} 
Define $a_n:=\sum_{\gamma\in \Gamma^W_Q(n,D)}\exp\big(\int_z^{\gamma z} \widetilde{F}\big)$, and $H(t):=\sum_{n\ge 2}e^{-nt}a_n$. By equation (\ref{xxxx}) we know that $H(t)$  converges if $t>\delta_\Gamma^F(Q,W,D)$ and diverges if $t<\delta_\Gamma^F(Q,W,D)$. With this notation we can rewrite (\ref{tp}) into 
\begin{align}\label{tpp}\sum_{\gamma\in \Gamma} \exp(\int_z^{\gamma z} \widetilde{F}+\widetilde{\phi}-s)\ge R_s H(s+\e),\end{align}
Taking $s=\delta_\Gamma^F(Q,W,D)-2\epsilon,$ we get  the divergence of the right hand side in (\ref{tpp}). This implies that $P(F+\phi)\ge \delta_\Gamma^F(Q,W,D)-2\epsilon,$  and therefore $$P(F+\phi)\ge \sup_D\delta_\Gamma^F(Q,W,D)-2\e.$$ 
Taking an increasing sequence $(Q_k)_k$ where $Q_0=Q$ and such that $(\sup_D\delta^F_\infty(Q_k,W,D))_k$ converges to $\delta^F_{\Gamma,\infty}$, we obtain $P(F+\phi)\ge \delta^F_{\Gamma,\infty}-2\e.$ Since $\e>0$ was arbitrary we are done.

\end{proof}



Our next result gives a formula for the measure theoretic pressure at infinity (see Definition \ref{defpresinf}) in terms of the topological pressure. 
\begin{theorem}\label{carla2} Let $F\in C_b(T^1M)$ be a uniformly continuous potential and $\phi\in C_0(T^1M)$ a strictly positive function. Then $$P_\infty(F)=\lim_{t\to\infty}P(F-t\phi)\ge \delta_{\infty,\Gamma}^F$$
\end{theorem}
\begin{proof} Define $G(t):=P(F-t\phi)$. Observe that $G$ is a non-decreasing function bounded below by $\delta_{\Gamma,\infty}^F$ (see Lemma \ref{entinf11}), therefore the limit $A:=\lim_{t\to\infty}G(t)$ is well defined. It also follows from Lemma \ref{entinf11} that $A\ge \delta_{\Gamma,\infty}^F$. 

We will first prove that $A\ge P_\infty(F)$. Let $(\nu_n)_n\subset \M(g)$ be a sequence converging vaguely to the zero measure and such that $$\lim_{n\to\infty}\big(h_{\nu_n}(g)+\int Fd\nu_n\big)=P_\infty(F).$$
By the definition of the topological pressure we know that $$h_{\nu_n}(g)+\int (F-t\phi)d\nu_n\le P(F-t\phi),$$
for every $n$. Note that $\lim_{n\to\infty}\big(h_{\nu_n}(g)+\int (F-t\phi)d\nu_n\big)=P_\infty(F)$, for every $t\in \R$. It follows that $P_\infty(F)\le P(F-t\phi)$, and therefore $P_\infty(F)\le A$.  

We will now prove that $A\le P_\infty(F)$. Let $\mu_n\in \M(g)$ be a measure such that
 \begin{align}\label{carla}h_{\mu_n}(g)+\int (F-n\phi)d\mu_n\ge P(F-n\phi)-\frac{1}{n}.\end{align} We claim that the sequence $(\mu_n)_n$ converges vaguely to the zero measure. If not we would have $\limsup_{n\to\infty}\int \phi d\mu_n\ge c$, for some $c>0$. Since $P(F-n\phi)\ge \delta_{\Gamma,\infty}^F$, this assumption contradicts inequality  (\ref{carla}). We conclude that $(\mu_n)_n$ converges vaguely to the zero measure. By sending $n\to\infty$ in (\ref{carla}) we get 
$$P_\infty(F)\ge \limsup_{n\to\infty}\big( h_{\mu_n}(g)+\int  F d\mu_n\big)\ge \limsup_{n\to\infty}\big( h_{\mu_n}(g)+\int (F-n\phi)d\mu_n\big)\ge A.$$
We conclude that $P_\infty(F)\ge A$.
\end{proof}

\begin{proof}[Proof of Theorem \ref{varpri}] In light of Theorem \ref{carla2} it is enough to prove that $$P_\infty(0)=h_{\infty}\le \delta_\infty=\delta^0_{\Gamma,\infty}.$$
By Theorem \ref{A} we know that any sequence $(\mu_n)_n\subset \M(g)$ converging vaguely to the zero measure satisfies 
$$\limsup_{n\to\infty}h_{\mu_n}(g)\le \delta_\infty.$$
This readily implies that  $\delta_\infty\le h_\infty$.


\end{proof}

We except that the formula $P_\infty(F)=\delta^F_{\infty, \Gamma}$  holds in complete generality. This would follow directly from Theorem \ref{carla2} and Conjecture \ref{conj1}.



\section{Applications to thermodynamic formalism} \label{6}
In this section we obtain several applications of Theorem \ref{i1} to the existence and stability of equilibrium states. For a suitable class of potentials (see Definition \ref{cE}) we will compute the first derivative of the pressure and show the continuity of equilibrium states. We will also prove that SPR potentials in $C_0(T^1M)$ form an open and dense subset. 

\subsection{Existence and stability of measure of maximal entropy}
 We start with a criterion for the existence of measures of maximal entropy. This result was also obtained in \cite{st} by different methods.  

\begin{theorem}[Criterion for existence of the measure of maximal entropy]\label{cme} Let $(M,g)$ be a pinched negatively curved manifold. Assume that $M$ is SPR, that is $\delta_\infty<h_{top}(g)$. Then the geodesic flow on $M$ admits a measure of maximal entropy. 
\end{theorem}
\begin{proof}
Let $(\mu_n)_n$ be a sequence of invariant probability measures such that $$h_{\mu_n}(g)>h_{top}(g)-\frac{1}{n}.$$ Recall that $\M_{\le1}(g)$ is compact with respect to the vague topology.  We will assume that $(\mu_n)_n$ converges in the vague topology (otherwise take a subsequence). Let $\mu$ be the vague limit of the sequence. By Theorem \ref{i1} we have that 
\begin{align*} h_{top}(g)=\limsup_{n\to \infty} h_{\mu_n}(g)&\le |\mu|h_{\mu/|\mu|}(g)+(1-|\mu|)\delta_\infty\\
&\le |\mu|h_{top}(g)+(1-|\mu|)\delta_\infty.
\end{align*}
Assume for a second that $\mu$ is not a probability measure, then we get 
$$ |\mu|h_{top}(g)+(1-|\mu|)\delta_\infty<  h_{top}(g),$$
which leads to a contradiction. We conclude that $\mu$ is a probability measure. By Theorem \ref{us} we obtain that 
$$h_{top}(g)=\limsup_{n\to\infty}h_{\mu_n}(g)\le h_\mu(g),$$
which proves the statement.
\end{proof}

 We remark that by \cite{op} the measure of maximal entropy is unique (see Theorem \ref{pps}). The same argument used in the proof of Theorem \ref{cme} gives us a slightly more general result.

\begin{theorem} \label{cme2} Let $(M,g)$ be a pinched negatively curved manifold. Let $(\mu_n)_n$ be a sequence of invariant probability measures such that $$\lim_{n\to\infty} h_{\mu_n}(g)=h_{top}(g).$$ Then the following statements hold.  
\begin{enumerate}
\item\label{1} Suppose that $M$ is SPR.  Then  $(\mu_n)_{n}$  converges  in the weak* topology to the measure of maximal entropy.
\item\label{2} Suppose that the geodesic flow on $M$ does not admit a measure of maximal entropy.  Then  $(\mu_n)_{n}$  converges vaguely to zero. In this case we  have  $\delta_\infty=h_{top}(g)$. 
\item\label{3} Suppose that the geodesic flow on $M$  admits a measure of maximal entropy. Then the accumulation points of  $(\mu_n)_{n}$  lies in the set $\{t\mu_{max}:t\in [0,1]\}$, where $\mu_{max}$ is the measure of maximal entropy.
\end{enumerate}
\end{theorem}
\begin{proof} Part (\ref{1}) follows directly from the proof of Theorem \ref{cme2}. We will now justify part (\ref{2}) and (\ref{3}). Let $\mu$ be a vague limit of $(\mu_n)_n$. Theorem \ref{i1} gives us that $$h_{top}(g)\le |\mu|h_{\mu/|\mu|}(g)+(1-|\mu|)\delta_\infty.$$
Observe that $\delta_\infty\le h_{top}(g)$. This implies that $|\mu|h_{\mu/|\mu|}(g)+(1-|\mu|)\delta_\infty\le h_{top}(g)$. Combining these inequalities we obtain that
\begin{align}\label{xxz} h_{top}(g)= |\mu|h_{\mu/|\mu|}(g)+(1-|\mu|)\delta_\infty.\end{align}
Assume for a second that the geodesic flow has not measure of maximal entropy, and that $\mu$ is not the zero measure. In this case we have $h_{\mu/|\mu|}(g)<h_{top}(g)$, which contradicts equation (\ref{xxz}).  This proves part (\ref{2}). If there exists a measure of maximal entropy and $\mu$ is not the zero measure, then by (\ref{xxz}) we conclude that $\mu/|\mu|$ is the measure of maximal entropy. This proves part (\ref{3}). 
\end{proof}

\begin{theorem}[Characterization of SPR manifolds]\label{charspr} $(M,g)$ is SPR  if and only if for every sequence $(\mu_n)\subset \M(g)$ satisfying $\lim_{n\to\infty}h_{\mu_n}(g)=h_{top}(g)$ we have that $(\mu_n)_n$ does not converge to the zero measure. 
\end{theorem}
\begin{proof} Suppose that $(M,g)$ is not SPR, that is $h_{top}(g)=\delta_\infty$. The variational principle at infinity  implies the existence of a sequence $(\mu_n)_n\subset \M(g)$ converging vaguely to the zero measure such that $\lim_{n\to\infty} h_{\mu_n}(g)=h_{top}(g)$. This proves one implication of the characterization, for the other direction simply use Theorem \ref{cme2}(\ref{1}). 
\end{proof}

\subsection{Existence and stability of equilibrium states} 
 We will now obtain similar results to Theorem \ref{cme}--\ref{charspr} in connection to the pressure of continuous potentials. Recall that $c(F)$ was introduced in  Definition \ref{Edef}. Our next result is the basic tool to obtain the existence of equilibrium states.

\begin{theorem}\label{AA} Let $(M,g)$ be a pinched negatively curved manifold. Let $(\mu_n)_n$ be a sequence of invariant probability measures converging vaguely to $\mu$ and $F\in C_b(T^1M)$. Then 
$$\limsup_{n\to\infty} \bigg(h_{\mu_n}(g)+\int Fd\mu_n\bigg)\le |\mu|\bigg(h_{\mu/|\mu|}(g)+\int Fd\mu/|\mu|\bigg)+(1-|\mu|)(\delta_\infty+c(F)).$$ 
\end{theorem}
\begin{proof} This result follows directly from combining Theorem \ref{i1} and Lemma \ref{E}.
\end{proof}

Unfortunately the constant $\delta_\infty+c(F)$ is not necessarily sharp.  In fact we have the following conjecture. 

\begin{conjecture}\label{conj1} Let $(M,g)$ be a pinched negatively curved manifold. Let $(\mu_n)_n$ be a sequence of invariant probability measures converging vaguely to $\mu$ and $F\in C_b(T^1M)$ a uniformly continuous potential. Then 
$$\limsup_{n\to\infty} \bigg(h_{\mu_n}(g)+\int Fd\mu_n\bigg)\le |\mu|\bigg(h_{\mu/|\mu|}(g)+\int Fd\mu/|\mu|\bigg)+(1-|\mu|)\delta^F_{\Gamma,\infty}.$$ 
\end{conjecture}

As mentioned right above Theorem \ref{carla2}, Conjecture \ref{conj1} implies that for every uniformly continous potential $F\in C_b(T^1M)$ we have that $P_\infty(F)=\delta^F_{\Gamma,\infty}$.  

In Section \ref{funspaces} we introduced the space of potentials that converges to $D$ at infinity, which we denoted by $C(T^1M,D)$. Recall that the space $\cH$ is defined as $$\cH=\bigcup_{E\in \R}C(T^1M,E).$$ 

\begin{lemma}\label{sprequi}Let $F\in C(T^1M,E)$. Then $\delta^F_{\Gamma,\infty}=\delta_\infty+E$. 
\end{lemma}
\begin{proof} Choose a compact set $Q\subset T^1\widetilde{M}$ such that $|F(x)-E|<\e$, for every $x\in p(Q)$. In particular $|\widetilde{F}(x)-E|<\e$, for every $x\in Q$. Let $W$ be an open relatively compact subset of $T^1\widetilde{M}$ intersecting $\widetilde{\Omega}$.  Observe that the Poincar\'e series $P(F,Q,W,D,s)$  can  be bounded below by $\sum_{\gamma\in \Gamma^W_{Q}(D)}\exp(-(s-E+\epsilon)d(x,\gamma x))$, and above by $\sum_{\gamma\in \Gamma^W_Q(D)}\exp(-(s-E-\epsilon)d(x,\gamma x))$ (up to multiplicative constants). From this we can conclude that $|\delta_\Gamma^F(Q,W,D)-(\delta_\infty+E)|<2\epsilon$. Since this computation works for every $Q_1$ satisfying $Q\subset Q_1$, we obtain  $$|\delta^F_{\Gamma,\infty}-(\delta_\infty+E)|<2\epsilon.$$ Since $\epsilon>0$ was arbitrary we conclude that $\delta^F_{\Gamma,\infty}=\delta_\infty+E$. 

\end{proof}

If $F\in C(T^1M,E)$, then $c(F)=E$. Lemma \ref{sprequi} implies that if $F\in \cH$, then $$\delta^F_{\Gamma,\infty}=\delta_\infty+c(F).$$  Using Theorem \ref{AA} we  conclude that Conjecture \ref{conj1} holds for potentials in $\cH$. We also observe that $F\in\cH$ is SPR if and only if $P(F)>\delta_\infty+c(F)$.

\begin{definition}\label{cE}  We say that a uniformly continuous potential $F\in C_b(T^1M)$ belongs to the class $\cE$ if $$P(F)>\delta_\infty+c(F).$$
\end{definition}
Theorem \ref{AA} and the strategy of the proof of Theorem \ref{cme} give us the following criterion for the existence of equilibrium states.

\begin{theorem}[Criterion for the existence of equilibrium states] \label{eeq} Let $(M,g)$ be a pinched negatively curved manifold. Then every potential in $\cE$ admits at least one equilibrium state. 
\end{theorem} 
An immediate consequence of Theorem \ref{eeq} is that if $F\in \cH$ is SPR, then $F$ admits an equilibrium state. Another consequence of Theorem \ref{eeq} is that  potentials with small oscillation admit an equilibrium state. A weaker result (replacing $c(F)$ with $\sup F$) was obtained in \cite[Theorem 6.11]{irv} for Schottky manifolds and H\"older potentials using symbolic methods. 
\begin{corollary}  Let $F\in C_b(T^1M)$ such that $$c(F)-\inf_x F(x)<h_{top}(g)-\delta_\infty.$$ Then $F$ admits an equilibrium state.
\end{corollary}
\begin{proof} Since $\inf_x F(x)\le c(F)$, we obtain that $\delta_\infty<h_{top}(g)$. Theorem \ref{cme} implies that the geodesic flow has a measure of maximal entropy. Denote by $\mu$ the measure of maximal entropy. Then $$P(F)\ge h_{\mu}(g)+\int Fd\mu\ge h_{top}(g)+\inf_x F(x)>c(F)+\delta_\infty.$$
Theorem \ref{eeq} implies that $F$ has an equilibrium state.
\end{proof}

 Our next result is the equivalent of Theorem \ref{cme2} for potentials in $\cH$ and its proof is identical to the one of Theorem \ref{cme2}.  We remark that $P(F)\ge \delta^F_{\Gamma,\infty}$; this is an important observation in order to prove Theorem \ref{eeq2} (this is the equivalent to inequality $h_{top}(g)\ge \delta_\infty$). 

\begin{theorem} \label{eeq2} Let $F\in \cH$ and $(\mu_n)_n$ be a sequence of invariant probability measures such that  $$\lim_{n\to\infty}\big(h_{\mu_n}(g)+\int Fd\mu_n\big)=P(F).$$ Then the following statements hold 
\begin{enumerate}
\item\label{1xx} If $F$ is SPR, then every convergent (in the vague topology) subsequence of $(\mu_n)_{n}$  converges in the weak* topology to an equilibrium state of $F$. If $F$ is H\"older continuous then $(\mu_n)_n$ converges in the weak*  topology to the (unique) equilibrium state of $F$.
\item\label{2xx} Suppose that $F$ does not admit any equilibrium state.  Then  $(\mu_n)_{n}$  converges vaguely to zero. In this case we  have  $P(F)=P_\infty(F)$. 
\item\label{3xx}  Suppose that $F$ does admit an equilibrium state.  Then the accumulation points of  $(\mu_n)_{n\in \N}$  lies in the set $$\{t\mu:t\in [0,1]\text{ and }\mu\text{ is an equilibrium state of $F$}\}.$$
\end{enumerate}
\end{theorem}


As in Theorem \ref{charspr} this result gives a characterization of SPR potentials in $\cH$. Our next result  follows directly from the proof of Theorem \ref{cme} (when using Theorem \ref{AA} instead of Theorem \ref{A}) and it is an essential ingredient in order to prove Theorem \ref{deriva} and Theorem \ref{conteq}.
\begin{theorem} \label{eeq3} Let $F\in \cE$ and $(\mu_n)_n$ be a sequence of invariant probability measures such that  $$\lim_{n\to\infty}\big(h_{\mu_n}(g)+\int Fd\mu_n\big)=P(F).$$ Then  $(\mu_n)_{n}$  converges to an equilibrium state of $F$. 
\end{theorem}

We emphasize that in the results presented above we require no higher regularity on $F$ than uniform continuity. In particular the theory developed in \cite{pps} does not apply.  A big difference with respect to more regular potentials is the lack of uniqueness of equilibrium states for continuous potentials. It can be proven that we can slightly $C^0$-perturb any potential $F\in C_b(T^1M)$ into a potential with uncountably many equilibrium states. A crucial ingredient for that result is Theorem \ref{us}, which allows us to identify subderivatives of the pressure at $F$ to its equilibrium states. This follows from the work of Israel \cite{is} and will be explained in details somewhere else. The next result was proven in \cite[Theorem 10]{v} for geometrically finite manifolds and potentials in $C_0(T^1M)$.

\begin{theorem}[First derivative of the pressure]\label{deriva} Let $F$ be a H\"older continuous potential  in $\cE$. For every $G\in C_b(T^1M)$ the following holds:
$$\frac{d}{dt}_{|t=0}P(F+tG)=\int Gd\mu_F,$$
where $\mu_F$ is the equilibrium state of $F$. 
\end{theorem} 
\begin{proof}  The proof of this result is identical to the one of \cite[Theorem 10]{v}.
\end{proof}


Using Theorem \ref{eeq3} we can prove the continuity of a family of equilibrium states for potentials in $\cE$.

\begin{theorem}[Continuity of equilibrium states]\label{conteq} Let $F$ and $G$ be H\"older potentials. Assume that for $t\in (-\e,\e)$, we have that $(F+tG)\in \cE$. Denote by $\mu_t$ to the equilibrium state of $F+tG$. Then the map $t\mapsto \mu_t$, is continuous in the weak* topology. 
\end{theorem}
We remark that $\mu_t$ exists by Theorem \ref{eeq3} and that it is well defined (unique equilibrium state) because $F+tG$ is H\"older continuous. Observe  that the hypothesis $P(H)>\delta_\infty+c(H)$ is an open condition in $C_b(T^1M)$. In other words $\cE\subset C_b(T^1M)$ is an open subset. In particular if $F\in \cE$ and $G\in C_b(T^1M)$, then there exists $\e>0$ such that for every $t\in (-\e,\e)$ we have that $F+tG\in \cE$.  
\begin{proof}   We now claim that $(\mu_t)_t$ converges in the weak* topology to $\mu_0$ as $t$ goes to zero. Notice that 
$$\lim_{t\to 0}\big(h_{\mu_t}(g)+\int F d\mu_t\big)=\lim_{t\to 0}\big(h_{\mu_t}(g)+\int (F+tG)d\mu_t\big)=\lim_{t\to 0}P(F+tG)=P(F).$$
Since $F$ is H\"older we know that $\mu_0$ is the unique equilibrium state of $F$. We now use Theorem \ref{eeq3} to conclude that $(\mu_t)_t$ converges to $\mu_0$ in the weak* topology. The same argument can be done replacing $0$ with a point $t_0\in (-\e,\e)$. We conclude that the map $t\mapsto \mu_t$ is continuous in the weak* topology.
\end{proof}

\subsection{Strongly positive recurrent potentials in $C_0(T^1M)$}\label{7.2} In this section we will prove that SPR potentials are open and dense in $C_0(T^1M)$. The same holds for potentials in $\cH$, but for simplicity we restrict our attention to potentials that vanish at infinity. This fact should be compared with what is known for SPR potentials in countable Markov shifts \cite{csa}. Recall that a potential $F\in C_0(T^1M)$ is strongly positive recurrent if $P(F)>\delta_\infty$. The family of SPR potentials in $C_0(T^1M)$ will be denoted by $\S$.

\begin{lemma} \label{spropen}$\S$ is an open subset of $C_0(T^1M)$.
\end{lemma}
\begin{proof} Let $F\in \S$ and define $r=P(F)-\delta_\infty$. Observe that if $||G-F||_0<r$, then $P(G)>P(F)-r=\delta_\infty$. 
\end{proof}
\begin{lemma}\label{sprdense} $\S$ is dense in $C_0(T^1M)$. 
\end{lemma}
\begin{proof} Let $F\in C_0(T^1M)$, we will prove that we can approximate $F$ by SPR potentials. If $F\in \S$ there is nothing to prove. We will assume that $F$ is not strongly positive recurrent, in other words, that $P(F)=\delta_\infty$. By Theorem \ref{aproxc0} we know that the pressure of $F$ can be approximated by the pressure over compact subsets. In particular, given $\e>0$, there exists a compact invariant subset $K$ and a measure $\mu$ supported on $K$ such that $$h_\mu(g)+\int Fd\mu>P(F)-\frac{\epsilon}{2}.$$
Choose a compactly supported continuous function $G$ such that $G(x)=\e$, for every $x\in K$, and $||G||_0=\e$. Observe that 
$$h_\mu(g)+\int (F+G)d\mu=\e+h_\mu(g)+\int Fd\mu\ge P(F)+\frac{\e}{2}>\delta_\infty.$$
In particular we obtain that $H=(F+G)\in C_0(T^1M)$ is strongly positive recurrent and $||H-F||_0\le \e$. Since $\e$ was arbitrary we obtain that $\S$ is dense in $C_0(T^1M)$.
\end{proof}

Combining Lemma \ref{spropen} and Lemma \ref{sprdense} we obtain the following result. 
\begin{theorem} The space of SPR potentials $\S$ is open and dense in $C_0(T^1M)$.
\end{theorem}


It worth mentioning that if Conjecture \ref{conj1} holds, then the results in this section automatically apply to the entire family of SPR potentials, and not simply to those in $C_0(T^1M)$, or $\cH$.

\section{An equidistribution result}\label{ld}

In this section we will obtain a large deviation estimate  (see Proposition \ref{ldpA})  and use that to prove Theorem \ref{equi} and Theorem \ref{fin1}.

 This section is motivated by the work of Pollicott \cite{pol} and Kifer \cite{kif}; we will check that their strategy extends to the non-compact case when suitable assumptions are satisfied. We will focus on large deviations upper bounds. We emphasize that large deviations lower bounds are more delicate and stronger assumptions on our potential and manifold would need to be imposed (for a detailed discusion we refer the reader to \cite[Chapter 8]{v2}). 

Let us first introduce some notation. The length of a periodic orbit of the geodesic flow  $\gamma$ is denoted by $l(\gamma)$. The periodic orbit $\gamma$ induces an invariant probability measure $\mu_\gamma\in \M(g)$; this is simply a normalization of the one dimensional Lebesgue measure on $\gamma$. We say that $\mu_\gamma$ is a \emph{periodic measure}. Let $\M_p$ be the set of periodic measures in $\M(g)$. Define 
$$\M_p(t)=\{\mu_\gamma: \gamma\text{ a periodic orbit such that }l(\gamma)\le t\},$$
and 
$$\M_p(W,t)=\{\mu_\gamma: \gamma\text{ a periodic orbit such that }l(\gamma)\le t\text{ and }\mu_\gamma(W)>0\}.$$

In this section the equilibrium state of a potential $F$ (if exists) will be denoted by $\mu_F$. This is simply $\mu_F=m_F/|m_F|$, where $m_F$ is the Gibbs measure of $F$ (see Section \ref{pretf}).


The \emph{Gurevich pressure} of a H\"older potential $F$ such that $P(F)>0$ is defined as 
$$P_{Gur}(F)=\lim_{t\to\infty}\frac{1}{t}\log \sum_{g\in \cM_p(W,t)}\exp(l(g)\int F dg),$$
where $W$ is a relatively compact open subset of $T^1M$ intersecting the non-wandering set of the geodesic flow. We emphasize that the initial definition of the Gurevich pressure in \cite{pps} differs from the one presented here; it is proved in \cite[Theorem 4.7]{pps} that both notions coincide. We also emphasize that the Gurevich pressure can be defined for potentials that do not satisfy $P(F)>0$, but we will not deal with that situation here. Note that we are only counting primitive periodic orbits, this definition and the one using all closed geodesics coincide (see \cite[Remark 9.13]{pps}). It is known that $$P_{Gur}(F)=P(F),$$
 for every  H\"older potential (see \cite[Theorem 1.1]{pps}).  We will summarize all this information in the following result. 
\begin{theorem}\label{c} Let $W$ be a relatively compact open subset of $T^1M$ intersecting the non-wandering set of the geodesic flow. Then
$$P(F)=P_{Gur}(F)=\lim_{t\to\infty}\frac{1}{t}\log \sum_{\mu_\gamma\in \cM_p(W,t)}\exp(l(\gamma)\int F d\mu_\gamma),$$
for every H\"older potential $F$ such that $P(F)>0$.
\end{theorem}

The following simple lemma will be used later.
\begin{lemma}\label{cc} Let $F\in C_b(T^1M)$ be uniformly continuous and such that $P(F)>0$. Then
$$\limsup_{t\to\infty}\frac{1}{t}\log \sum_{\mu_\gamma\in \cM_p(W,t)}\exp(l(\gamma)\int F d\mu_\gamma)\le P(F).$$
\end{lemma}
\begin{proof} Since $F\in C_b(T^1M)$ is uniformly continuous, for every $\e>0$ there exists a H\"older potential $F_\e$ such that $||F-F_\e||_0<\e$. We consider $\e>0$ sufficiently small such that $P(F_\e)>0$. Note that 
\begin{align*}\sum_{\mu_\gamma\in \cM_p(W,t)}\exp(l(\gamma)\int F d\mu_\gamma)&\le \sum_{\mu_\gamma\in \cM_p(W,t)}\exp(\e l(\gamma))\exp(l(\gamma)\int F_\e d\mu_\gamma)\\ &\le \exp(\epsilon t) \sum_{\mu_\gamma\in \cM_p(W,t)}\exp(l(\gamma)\int F_\e d\mu_\gamma).\end{align*}
Therefore 
$$\limsup_{t\to\infty}\frac{1}{t}\log \sum_{\mu_\gamma\in \cM_p(W,t)}\exp(l(\gamma)\int F d\mu_\gamma)\le\e+P(F_\e).$$
Here we used Theorem \ref{c} for $F_\e$. Finally observe that $\e+P(F_\e)\le 2\e+P(F)$, but $\e$ was an arbitrary sufficiently small positive number.
\end{proof}



\begin{definition}[G-pressure]\label{gpre} Given $G\in C_0(T^1M)$ we define $Q_G:C_0(T^1M)\to \R,$ by the formula 
$$Q_G(f)=P(G+f)-P(G).$$ 
\end{definition}
\begin{definition}[Rate function]\label{rf} Given $\mu\in \cM_{\le 1}(g)$ we define $$I_G(\mu)=\sup_{f\in C_0(T^1M)}\big\{\int fd\mu-Q_G(f)\big\}.$$
\end{definition}

Observe that by definition (take $f=0$) we know $I_G(\mu)\ge 0$. If the choice of function $G$ is clear we will use the notation $Q$ (resp. $I$) instead of $Q_G$ (resp. $I_G$). 

Definitions \ref{gpre} and \ref{rf} are the starting point of our analysis (in the compact case they were introduced in \cite{pol}).  We remark that in the non-compact case the choice of domain for the G-pressure and the rate function  might not be  canonical. For convenience we defined the rate function in the space of sub-probability measures $\M_{\le 1}(g)$; once this decision is made the domain of $Q_G$ must be contained in $C_0(T^1M)$ (to ensure duality). One good reason to justify this choice  is that in  Proposition \ref{ldpA} we want to consider arbitrary closed subsets, instead of compact subsets (by a standard Contraction Principle we  could bypassed this issue if the sequence $(\cM_p(W,t))_t$ is exponentially tight, but we prefer not to use this approach, see  \cite[Theorem 4.2.1]{dz}). 

\begin{lemma} The fuction $I_G$ is lower semicontinuous. 
\end{lemma}
\begin{proof} We want to prove that the set $I_G^{-1}( (t,\infty])$ is open. Pick a measure $\mu\in\cM_{\le 1}(g)$ such that $I_G(\mu)>t$. This implies that there exists $f\in C_0(T^1M)$ such that $\int fd\mu-Q_G(f)>t$. By the definition of the vague topology the map $\mu\mapsto \int fd\mu-Q_G(f)$, is continuous. In particular the inequality $\int fd\nu-Q_G(f)>t$, holds for measures $\nu$ in a neighorhood of $\mu$. This implies that $I_G(\nu)>t$, for every $\nu$ in this neighborhood.  
\end{proof}

The following proposition follows from the upper semicontinuity of the entropy map (see Theorem \ref{us}) and the proof of \cite[Theorem 9.12]{wa} after some mild modifications. 
\begin{proposition} \label{ldpentropy} For every $\mu\in\cM(g)$ the following formula holds
$$h_\mu(g)=\inf_{f\in C_c(T^1M)} \big\{P(f)-\int fd\mu\big\}.$$
\end{proposition}
\begin{remark} Observe that by definition of the topological pressure we have $h_\mu(g)+\int fd\mu\le P(f)$. In particular $$h_\mu(g)\le \inf_{f\in C_0(T^1M)}\big\{P(f)-\int fd\mu\big\}.$$ Proposition \ref{ldpentropy} gives the lower bound $h_\mu(g)\ge \inf_{f\in C_c(T^1M)}\big\{P(f)-\int fd\mu\big\}$. In particular $h_\mu(g)\ge \inf_{f\in C_0(T^1M)}\big\{P(f)-\int fd\mu\big\}$. We conclude that the formula in Proposition \ref{ldpentropy}  holds when replacing $C_c(T^1M)$ by $C_0(T^1M)$. 
\end{remark}

The following lemma was taken from \cite{pol}. For completeness we provide a proof. 

\begin{lemma}  \label{lemrate} Let $G\in C_0(T^1M)$. Then for every $\mu\in\cM(g)$ the following formula holds
$$I_G(\mu)=P(G)-(h_\mu(g)+\int Gd\mu).$$
\end{lemma}
\begin{proof}
Observe that 
\begin{align*}
I_G(\mu)=&\sup_{f\in C_0(T^1M)}\big\{\int (G+f)d\mu-P(G+f)\big\} + P(G)-\int G d\mu\\
=&-\inf_{f\in C_0(T^1M)}\big\{P(G+f)-\int (G+f)d\mu\big\}+ P(G)-\int G d\mu\\
=&-\inf_{g\in C_0(T^1M)}\big\{P(g)-\int gd\mu\big\}+ P(G)-\int G d\mu\\
=&P(G)-(h_\mu(g)+\int Gd\mu).
\end{align*}
\end{proof}


\subsection{Upper bound for closed sets} 
Our next result follows closely the proof of \cite[Theorem 1 (i)]{pol} (also see \cite{kif}). The main difference is that we work with sub-probability measures and our formula for the Gurevich pressure only works when $P(F)>0$. Let $W$ be a relatively compact open subset of $T^1M$ intersecting the non-wandering set of the geodesic flow. 

\begin{proposition} \label{ldpA} Let $G$ be a H\"older potential in $C_0(T^1M)$ such that $P(G)>0$. Let $\cK$ be a closed subset of $\cM_{\le 1}(g)$. Suppose there exists $\nu\in \cK\cap \M(g)$ such that $h_\nu(g)+\int Gd\nu\ge 0$. Then we have $$\limsup_{t\to\infty}\frac{1}{t}\log \bigg( \frac{\sum_{\mu_\tau\in \cM_p(W,t)\cap \cK} \exp(l(\tau)\int G d\mu_\tau) }{\sum_{\mu_\tau\in \cM_p(W,t)} \exp(l(\tau)\int Gd\mu_\tau )} \bigg)\le -\beta,$$
where $\beta=\inf_{\mu\in \cK} I_G(\mu)$.
\end{proposition}

We remark that the condition $P(G)>0$ and the assumption that $\cK$ contains a probability measure such that $h_\nu(g)+\int Gd\nu\ge 0$ can be removed if we sum over sets of the form $$\M_p(W,t,t-c)=\{\mu_\gamma: \gamma\text{ periodic such that }l(\gamma)\in [t-c,t]\text{ and }\mu_\gamma(W)>0\},$$
where $c$ is sufficiently large (see \cite[Chapter 8]{v2}). The statement of Proposition \ref{ldpA} is enough for our purposes--we want to obtain an equidistribution result when summing over $\M_p(W,t)$, and not just $\M_p(W,t, t-c)$. 

\begin{proof}  From now on we will use the notation $I=I_G$, and $Q=Q_G$.
We remark that since $K\subset \cM_{\le1}(g)$ is closed, it is immediately compact. 
For every $\epsilon>0$ we have
$$\cK\subset\{\mu\in \cM_{\le 1}(g): I(\mu)>\beta-\epsilon\}.$$
By the definition of $I(\mu)$ we have that  $$\{\mu\in \cM_{\le 1}(g): I(\mu)>\beta-\epsilon\}=\bigcup_{f\in C_0(T^1M)}\{\mu\in\cM_{\le1 }(g): \int f d\mu-Q(f)>\beta-\epsilon\}.$$
Define $\cV_f =\{\mu\in\cM_{\le 1}(g): \int f d\mu-Q(f)>\beta-\epsilon\}$, and observe that $\cV_f$ is open in the   vague topology (since $f\in C_0(T^1M)$). By the compactness of $\cK$ we obtain a  finite subcover $K\subset \bigcup_{i=1}^N\cV_{f_i}$. Define $\cU_{f_i}(t)=\cV_{f_i}\cap \cM_p(W,t)\cap \cK$. We have the following inequality

\begin{align*}
\frac{\sum_{\mu_\tau\in \cM_p(W,t)\cap\cK} \exp(l(\tau)\int G d\mu_\tau) }{\sum_{\mu_\tau\in \cM_p(W,t)} \exp(l(\tau)\int Gd\mu_\tau )} &\le \sum_{i =1}^N\frac{\sum_{\mu_\tau\in \cU_{f_i}(t)} \exp(l(\tau)\int G d\mu_\tau) }{\sum_{\mu_\tau\in \cM_p(W,t)} \exp(l(\tau)\int Gd\mu_\tau )} \\
& = \sum_{i=1}^N  \frac{\sum_{\mu_\tau\in \cU_{f_i}(t)} \exp(l(\tau)\int (G +f_i) d\mu_\tau)e^{-l(\tau)\int f_id\mu_\tau}}{\sum_{\mu_\tau\in \cM_p(W,t)} \exp(l(\tau)\int Gd\mu_\tau) } \\
&\le \sum_{i=1}^N \frac{\sum_{\mu_\tau\in \cU_{f_i}(t)} \exp(l(\tau)\int (G +f_i) d\mu_\tau)e^{-l(\tau)(Q(f_i)+\beta-\epsilon)}}{\sum_{\mu_\tau\in \cM_p(W,t)} \exp(l(\tau)\int Gd\mu_\tau) }. \\
\end{align*}
Define $H_i:=G+f_i-P(G+f_i)+\e$. With this notation and the formula $Q(f_i)=P(G+f_i)-P(G)$ our last inequality becomes 
\begin{align}\label{ccc}
\frac{\sum_{\mu_\tau\in \cM_p(W,t)\cap\cK} \exp(l(\tau)\int G d\mu_\tau) }{\sum_{\mu_\tau\in \cM_p(W,t)} \exp(l(\tau)\int Gd\mu_\tau )}  &\le \sum_{i=1}^N \frac{\sum_{\mu_\tau\in \cU_{f_i}(t)} \exp(l(\tau)\int H_i d\mu_\tau)e^{l(\tau)(P(G)-\beta)}}{\sum_{\mu_\tau\in \cM_p(W,t)} \exp(l(\tau)\int Gd\mu_\tau) }. 
\end{align}
Observe that $P(H_i)=\e>0$, in particular Lemma \ref{cc} implies that $$\limsup_{t\to\infty}\frac{1}{t}\log \sum_{\mu_\tau\in \cU_{f_i}(t)} \exp(l(\tau)\int H_i d\mu_\tau)\le P(H_i)=\e.$$
Recall that $l(\gamma)\le t$ whenever $\mu_\gamma\in \M_p(W,t)$. Moreover 
$$P(G)-\beta=P(G)-\inf_{\mu\in \cK}I(\mu)\ge P(G)-I(\nu)=h_\nu(g)+\int Gd\nu\ge 0.$$ Therefore
\begin{align*} \limsup_{t\to \infty}\frac{1}{t}\log \bigg(\frac{\sum_{\mu_\tau\in \cU_{f_i}(t)} \exp(l(\tau)\int H_id\mu_\tau)e^{l(\tau)(P(G)-\beta)}}{\sum_{\mu_\tau\in \cM_p(W,t)} \exp(l(\tau)\int Gd\mu_\tau) }\bigg) &\le \epsilon+(P(G)-\beta)-P(G)\\
&=-\beta+\e. 
\end{align*}
Here we used the fact that $\lim_{t\to\infty}\frac{1}{t}\log \sum_{\mu_\tau\in \cM_p(W,t)} \exp(l(\tau)\int Gd\mu_\tau)$, is well defined and equal to $P(G)$ (see Theorem \ref{c}). Combining this with the inequalities $$\limsup_{t\to\infty}\frac{1}{t}\log \big(\sum_{i=1}^N a_i(t)\big)\le\max_{i\in\{1,..,N\}} \limsup_{t\to\infty}\frac{1}{t}\log a_i(t),$$
and  (\ref{ccc}), we obtain that
\begin{align*}
\limsup_{t\to\infty} \frac{1}{t} \log \bigg( \frac{\sum_{\mu_\tau\in \cM_p(W,t)\cap\cK} \exp(l(\tau)\int G d\mu_\tau) }{\sum_{\mu_\tau\in \cM_p(W,t)} \exp(l(\tau)\int Gd\mu_\tau )}\bigg)\le -\beta+\epsilon,
\end{align*}
but $\epsilon>0$ was arbitrary.

\end{proof}

Our next lemma is the reason why we have to restrict to SPR potentials in  Theorem \ref{equi}.

\begin{lemma} \label{ldplem} Let $H\in C_0(T^1M)$ be a H\"older SPR potential. Then the only measure $\mu\in\cM_{\le 1}(g)$ satisfying the equation $$I_H(\mu)=0,$$ is the equilibrium state of $H$.  
\end{lemma}
\begin{proof} Recall that $I_H(\mu)\ge 0$, for every $\mu\in\cM_{\le 1}(g)$. If $I_H(\mu)=0$,  we must have 
$$\int fd\mu-P(H+f)+P(H)\le 0,$$
for every $f\in C_0(T^1M)$. In particular we have $t\int fd\mu\le P(H+tf)-P(H)$. Suppose $t>0$, divide by $t$ and take $t\to 0$. By Theorem \ref{deriva} we obtain $\int fd\mu\le \int fd\mu_H$, where $\mu_H$ is the equilibrium state of $H$. Since $f$ was an arbitrary function in $C_0(T^1M)$ we necessarily have $\mu=\mu_H$. 
\end{proof}
 If we assume that $H\in C_0(T^1M)$ is not strongly positive recurrent, then $$I_H(\mu_0)=\sup_{f\in C_0(T^1M)}\{P(H)-P(H+f)\}=P(H)-\inf_{f\in C_0(T^1M)} P(H+f),$$
where $\mu_0$ is the zero measure. By Theorem \ref{carla2} for $F=0$ we know that $$\inf_{f\in C_0(T^1M)} P(H+f)=\inf_{g\in C_0(T^1M)}P(g)=\delta_\infty.$$ Since $H$ is not strongly positive recurrent we get that $I_H(\mu_0)=0$. In this case the conclusion of Lemma \ref{ldplem} does not hold.

\begin{remark}\label{carlane} Let $H\in C_0(T^1M)$ be a H\"older SPR potential. Observe that if $\cK\subset \M_{\le 1}(g)$ is a closed set that does not contain the equilibrium state of $H$ then $\inf_{\mu\in \cK}I_G(\mu)>0$. Argue by contradiction and assume that there exists $(\mu_n)_n\subset \cK$ such that $\lim_{n\to\infty} I_G(\mu_n)=0$. By the compactness of $\cK$  we can assume that $(\mu_n)_n$ converges in the vague topology to some $\mu\in \cK$. Since $I_G$ is lower semicontinuous we conclude that $$0=\liminf_{n\to\infty}I_G(\mu_n)\ge I_G(\mu).$$ Therefore $I_G(\mu)=0$, which by Lemma \ref{ldplem} implies that $\cK$ contains the equilibrium state of $G$, a contradiction. 
\end{remark}

As before we assume that $W$ is a relatively compact open subset of $T^1M$ intersecting the non-wandering set of the geodesic flow. Define $$m_p(W,t)=\sum_{\gamma\in \M_p(W,t)}\exp(l(\gamma)\int F d\mu_\gamma),$$ and 
$$\mu_p(W,t)=\frac{1}{m_{p}(W,t)}\sum_{\mu_\gamma\in \M_p(W,t)}\exp(l(\gamma)\int F d\mu_\gamma)\mu_\gamma,$$
whenever $m_p(W,t)$ is positive. By definition $\mu_p(W,t)$ is an invariant probability measure. We first obtain the following equidistribution result which is a standard application of a large deviation estimate like Proposition \ref{ldpA} (see \cite{pol}). 

\begin{theorem}\label{equi2} Let $F\in C_0(T^1M)$ be a H\"older SPR potential such that $P(F)>0$. Then the sequence $(\mu_p(t))_t$ converges to the equilibrium state $\mu_F$ of $F$ in the weak* topology. 
\end{theorem}
\begin{proof}  Since $(m_p(W,t))_t$ and $\mu_F$ are probability measures it is enough to prove that for every $H\in C_0(T^1M)$ we have that $$\lim_{t\to\infty}\int Hd\mu_p(W,t)=\int Hd\mu_F,$$ (see Theorem \ref{weakvag}). 

Fix $H\in C_0(T^1M)$ such that $\{\int Hd\mu:\mu\in \M(g)\}$ has more than one element (otherwise there is nothing to prove). Choose $\eta\in \M(g)$ such that $\epsilon_0:=|\int Hd\eta- \int H d\mu_F|>0$. By definition $h_{\mu_F}(g)+\int Fd\mu_F=P(F)>0$. Choose $t_0$ sufficiently small such that $\eta_1:=t_0\eta+(1-t_0)\mu_F$ satisfies $h_{\eta_1}(g)+\int Fd\eta_1>0$. Observe that $|\int H d\eta_1-\int Hd\mu_F|=t_0|\int H d\eta-\int H d\mu_F|=t_0\epsilon_0>0$. This implies that the set
$$\cK_\e:=\{\mu\in \M_{\le 1}(g): |\int Hd\mu-\int Hd\mu_F|\ge \epsilon\},$$
contains a measure $\nu\in \M(g)$ such that $h_{\nu}(g)+\int F d\nu>0$, provided $\epsilon\in (0,t_0\e_0)$. Since $H\in C_0(T^1M)$ the set $\cK_\e$ is closed. Let 
$$A=\sum_{\mu_\gamma\in \M_p(W,t)\cap \cK_\e}e^{l(\gamma)\int Fd\mu_\gamma}\big(\int Hd\mu_\gamma-\int Hd\mu_F\big),$$
and 
$$B=\sum_{\mu_\gamma\in \M_p(W,t)\cap \cK_\e^c}e^{l(\gamma)\int Fd\mu_\gamma}\big(\int Hd\mu_\gamma-\int Hd\mu_F\big).$$
Observe that $A\le 2 \sum_{\mu_\gamma\in \M_p(W,t)\cap \cK_\e}e^{l(\gamma)\int Fd\mu_\gamma}||H||_0$, and $B\le \epsilon m_p(W,t)$. Then
\begin{align*}\label{y}|\int Hd\mu_p(W,t)-\int Hd\mu_F|\le \frac{|A|+|B|}{m_p(W,t)}\le 2\frac{ \sum_{\mu_\gamma\in \M_p(W,t)\cap \cK_\e}e^{l(\gamma)\int Fd\mu_\gamma}}{m_p(W,t)}+\e.\end{align*}
For $\e\in (0,t_0\e_0)$ the set $\cK_\e$ satisfies all the hypotheses of Proposition \ref{ldpA}. In particular we have that  
$$\limsup_{t\to\infty}\frac{1}{t}\log \bigg( \frac{\sum_{\mu_\gamma\in \cM_p(W,t)\cap \cK_\e} \exp(l(\gamma)\int F d\mu_\gamma) }{\sum_{\mu_\gamma\in \cM_p(W,t)} \exp(l(\gamma)\int Fd\mu_\gamma)} \bigg)\le -\beta_\e,$$
where $\beta_\e:=\inf_{\mu\in \cK_\e} I_F(\mu)$. Since $\mu_F$ does not belong to $\cK_\e$ we know that $\beta_\e>0$ (see Remark \ref{carlane}). In particular we obtain that 
$$\limsup_{t\to\infty}|\int Hd\mu_p(W,t)-\int Hd\mu_F|\le \e.$$
Since $\e$ was an arbitrary number in $(0,t_0\e_0)$, we conclude $$\lim_{t\to\infty}\int Hd\mu_p(W,t)=\int H d\mu_F.$$

\end{proof}

We will need an important equidistribution result obtained in \cite[Theorem 9.11 (2)]{pps}. When $M$ is convex-cocompact this result was obtained by Bowen \cite{bo9}, \cite{bo8},\cite{b73} and generalized by Roblin \cite[Theorem 5.1.1]{ro} when $F=0$. We remark that in the notation of \cite{pps} the vague convergence (resp. weak* convergence) is called weak* convergence (resp. vague convergence). In our language  \cite[Theorem 9.11 (2)]{pps} becomes
\begin{theorem}\label{equipps} Let $(M,g)$ be a pinched negatively curved manifold and assume that the geodesic flow is topologically mixing. Let $F\in C_b(T^1M)$ be a H\"older potential which admits an equilibrium state $\mu_F$ and such that $P(F)>0$. Then 
$$P(F)te^{-P(F)t}\sum_{\mu_\gamma\in \M_p(t)}\exp(l(\gamma)\int Fd\mu_\gamma)\mu_\gamma,$$
converges in the vague topology to $\mu_F$. 
\end{theorem}

In the case of geometrically finite manifolds Theorem \ref{equipps} can be upgraded to conclude that the measures converges in the weak* topology. For compleness we state such result, which was proven in \cite[Theorem 9.16]{pps}. For $F=0$ this was obtained by Roblin \cite[Theorem 5.2]{ro}. 

\begin{theorem}\label{yess} Suppose that $(M,g)$ is geometrically finite. Let $F\in C_b(T^1M)$ be a H\"older potential with positive pressure which admits an equilibrium state $\mu_F$. Then 
$$P(F)te^{-P(F)t}\sum_{\mu_\gamma\in \M_p(t)}\exp(l(\gamma)\int Fd\mu_\gamma)\mu_\gamma,$$
converges in the weak* topology to $\mu_F$. In particular the following holds
$$\sum_{\mu_\gamma\in \M_p(t)}e^{l(\gamma)\int F d\mu_\gamma}\sim \frac{e^{P(F)t}}{P(F)t}.$$
\end{theorem}

The counting statement in Theorem \ref{yess} was first obtained in the compact case (for $F=0$) by Margulis \cite{marg} and it is usually called the prime geodesic theorem (see also \cite{papo}). Our next result gives new information in  case that $M$ is not geometrically finite. We believe the assumption $F\in C_0(T^1M)$ can be removed, this would follow from an affirmative answer to Conjecture \ref{conj1}. The assumption that $F$ is SPR seems to be  more substantial, but it could still be the case that it is not really necessary (as in the geometrically finite case). 

\begin{theorem}\label{fin}  Let $(M,g)$ be a pinched negatively curved manifold and assume that the geodesic flow is topologically mixing. Let $F\in C_0(T^1M)$ be a H\"older SPR potential such that $P(F)>0$. Then 
$$\sum_{\mu_\gamma\in \M_p(W,t)}e^{l(\gamma)\int F d\mu_\gamma}\sim \frac{e^{P(F)t}}{P(F)t},$$
for every open relatively compact subset $W\subset T^1M$ such that $W\cap \Omega\ne \emptyset$.
\end{theorem}
\begin{proof} Pick a point $x\in W\cap \Omega$ and take a small ball $U$ around $x$ such that $U\subset W$. Let $H\in C_c(T^1M)$ be a non-negative potential such that $\supp H\subset U$ and such that $H(x)>0$. Since the support of $\mu_F$ (and of every equilibrium state of a H\"older potential) is $\Omega$ we have that $\int Hd\mu_F>0$. For $\mu\in \M(g)$ we have $\int Hd\mu\ne 0$ only if $\mu(U)>0$, in particular 
$$\sum_{\mu_\gamma\in \M_p(t)}e^{l(\gamma)\int Fd\mu_\gamma}\int Hd\mu_\gamma=\sum_{\mu_\gamma\in \M_p(W,t)}e^{l(\gamma)\int Fd\mu_\gamma}\int H d\mu_\gamma.$$
By Theorem \ref{equipps} we have that 
$$\lim_{t\to\infty} P(F)te^{-P(F)t}\sum_{\mu_\gamma\in \M_p(t)}e^{l(\gamma)\int Fd\mu_\gamma}\int Hd\mu_\gamma=\int Hd\mu_F. $$
Therefore
$$\lim_{t\to\infty} P(F)te^{-P(F)t}m_p(W,t)\bigg(\frac{1}{m_p(W,t)}\sum_{\mu_\gamma\in \M_p(W,t)}e^{l(\gamma)\int Fd\mu_\gamma}\int Hd\mu_\gamma\bigg)=\int Hd\mu_F.$$
Now use Theorem \ref{equi2} to conclude 
$$\lim_{t\to\infty} P(F)te^{-P(F)t}m_p(W,t)\int Hd\mu_F=\int Hd\mu_F.$$
Since $\int Hd\mu_F>0$ we conclude that $$\lim_{t\to\infty} P(F)te^{-P(F)t}m_p(W,t)=1,$$
as desired.

\end{proof}

As a corollary of Theorem \ref{fin} we obtain a version of the prime geodesic theorem for strongly positive recurrent manifolds.

\begin{corollary}\label{fincor} Let $(M,g)$ be strongly positive recurrent and assume its geodesic flow is topologically mixing. Then 
$$\#\{\gamma\text{ primitive periodic orbit }|\text{ } l(\gamma)\le t\text{ and }\gamma\cap W\ne \emptyset\} \sim \frac{e^{\delta_\Gamma t}}{\delta_\Gamma t},$$
for every open relatively compact subset $W\subset T^1M$ such that $W\cap \Omega\ne \emptyset$.
\end{corollary}

It would be interesting to obtain error terms for this counting result. The class of strongly positive recurrent manifolds is a natural generalization of hyperbolic geometrically finite manifolds, where this problem has been successfully studied (see \cite[Theorem 1.2]{mmo}). As in \cite{mmo}, this would follow from the exponential decay of correlations of the measure of maximal entropy. The exponential decay of correlations of the measure of maximal entropy was proven for geometrically finite hyperbolic manifolds with sufficiently large critical exponent by A. Mohammadi and H. Oh in \cite{mo}. Because of the similarities between the ergodic theory of countable Markov shifts and the geodesic flow, we suspect the exponential decay of correlations should hold for every SPR manifold (see \cite{csa})--this is certainly an important open problem. 

We finish with the construction of a SPR manifold with infinitely many geodesics of the same length.  This justifies that in the geometrically infinite setting--even within the class of SPR manifolds--we can not remove the condition of counting only geodesics that intersect a  compact part of $M$. 

Let $\Gamma_1$ and $\Gamma_2$ be two discrete torsion free subgroups of $Iso(\widetilde{M})$. We say that $\Gamma_1$ and $\Gamma_2$ are in Schottky position if there exist disjoint compact sets $U_{\Gamma_1}, U_{\Gamma_2}\subset \widetilde{M}\cup\partial_\infty\widetilde{M}$, such that for every $\gamma_1\in \Gamma_1\setminus \{id\}$ and $\gamma_2\in \Gamma_2\setminus \{id\}$, we have
$$\gamma_1((\widetilde{M}\cup\partial_\infty\widetilde{M})\setminus U_{\Gamma_1}\subset U_{\Gamma_1}\text{, and }\gamma_2((\widetilde{M}\cup\partial_\infty\widetilde{M})\setminus U_{\Gamma_2})\subset U_{\Gamma_2}.$$
The ping-pong lemma implies that the group generated by $\Gamma_1$ and $\Gamma_2$, say $\Gamma=\langle \Gamma_1,\Gamma_2\rangle$, is isomorphic to the free product $\Gamma_1*\Gamma_2$. The Klein-Maskit combination theorem implies that the group $\Gamma$ is discrete and torsion free; $\Gamma$ is called the Schottky combination of $\Gamma_1$ and $\Gamma_2$. We use the notation $\delta_{ \Gamma_i,\infty}$ to denote the topological entropy at infinity of the geodesic flow on $\widetilde{M}/\Gamma_i$. 

\begin{example}\label{finrem}
Let $M_0=\H^2/\Gamma_0$ be a closed hyperbolic surface, and $\phi:\Gamma_0\to\Z$ a surjective homeomorphism. The hyperbolic surface $\H^2/\Gamma^0_1$, where $\Gamma^0_1=\ker\phi$, is a $\Z$-cover of $M_0$. We identify $\Gamma^0_1$ with a subgroup $\Gamma_1$ of $PSL(2,\C)$ and define $M_1=\H^3/\Gamma_1$. Let $M_2=\H^3/\Gamma_2$ be a geometrically finite hyperbolic 3-manifold with at least one rank 2 cusp and infinite volume. In this case the domain of discontinuity of $\Gamma_2$ is non-empty and it is possible to find a fundamental domain for the action of $\Gamma_2$ (on its domain of discontinuity) with non-empty interior. Maybe after a conjugation of $\Gamma_2$ we can assume that $\Gamma_1$ and $\Gamma_2$ are in Schottky position. Denote by $\Gamma\leqslant PSL(2,\C)$ the Schottky combination  of $\Gamma_1$ and $\Gamma_2$, and $M=\H^3/\Gamma$. In this situation we can use \cite[Theorem 7.18]{st}, which says that $\delta_{\infty, \Gamma}=\max\{\delta_{ \Gamma_1,\infty},\delta_{ \Gamma_2,\infty}\}.$ It is proved in \cite{rv} that $\delta_{ \Gamma_2,\infty}=\max_{\P}\delta_\P$, where the maximum runs over the critical exponent of the parabolic subgroups of $\Gamma_2$. Since $M_2$ has a rank 2 cusp, and the critical exponent of a rank 2 cusp is 1, we conclude that $\delta_{ \Gamma_2,\infty}=1$. Since every parabolic subgroup of $PSL(2,\C)$ is of divergence type we can use \cite[Proposition 2]{dop} to conclude that $1=\delta_{ \Gamma_2,\infty}<\delta_{\Gamma_2}$. Observe that the critical exponent of $\Gamma_1$ is the same as the critical exponent of $\Gamma_1^0$. Since $\Gamma_1^0$ is a normal subgroup of $\Gamma_0$, and $\Gamma_0$ is cocompact, we get that $\delta_{\Gamma_1^0}=\delta_{\Gamma_0}=1$. We conclude that $\delta_{\Gamma_1}=1$ and that $\delta_{ \Gamma_1,\infty}\le 1$. Finally, putting all this together we get that 
$$\delta_{ \Gamma,\infty}=\max\{\delta_{ \Gamma_1,\infty},\delta_{ \Gamma_2,\infty}\}=1<\delta_{\Gamma_2}\le \delta_{\Gamma}.$$
We conclude that $M$ is a SPR hyperbolic 3-manifold. Since $\Gamma_1\leqslant \Gamma$, it follows that $M$ has infinitely many geodesics of some length (the same that happens for $\H^2/\Gamma_1^0$, and therefore for $\H^3/\Gamma_1$). It worth mentioning that $\pi_1(M)=\Gamma$ is not finitely generated.
\end{example}

\end{document}